\documentclass[a4paper,10pt]{article}
\usepackage{authblk}
\usepackage{graphicx}

\usepackage{amsmath,amssymb,latexsym, amsthm, amsfonts}
\usepackage{url} 

\usepackage[left=3cm,right=2cm,top=3cm,bottom=2cm]{geometry}

\usepackage[all]{xypic}

\usepackage[all,2cell,ps]{xy}

\newcommand{\luk}{\L u\-ka\-si\-e\-w\-icz}

\newtheorem{theorem}{Theorem}[section]
\newtheorem{lemma}[theorem]{Lemma}
\newtheorem{corollary}[theorem]{Corollary}
\newtheorem{proposition}[theorem]{Proposition}
\newtheorem{fact}{Fact}

\theoremstyle{definition}
\newtheorem{definition}[theorem]{Definition}
\newtheorem{remark}[theorem]{Remark}
\newtheorem{example}[theorem]{Example}
\newcommand{\msim}{\mathord{\sim}}
\newcommand{\ninv}{\mathord{\sim}}  
\newcommand{\finv}{\mathord{-}}  

\newcommand{\To}{\Rightarrow}
\newcommand{\TTo}{\Leftrightarrow}
\newcommand{\Ln}{\L_{n+1}}
\newcommand{\Lns}{\L^*_{n+1}}
\newcommand{\bL}{\textrm{\bf \L}}
 
\newcommand{\co}{{\bf c}}

\newcommand{\ntwrt}[1]{non-trivial with respect to \ensuremath{#1}}

\begin{document}

\title{On the expressive power of \luk's square operator}

\author{Marcelo E. Coniglio$^{1}$, Francesc Esteva$^{2}$, Tommaso Flaminio$^{2}$, Lluis Godo$^{2}$\\
\small $^1$Centre for Logic, Epistemology and the History of Science - CLE, and\\
\small Institute of Philosophy and the Humanities - IFCH\\
\small University of Campinas, Brazil\\
\small $^2$Artificial Intelligence Research Institute (IIIA) - CSIC\\ Barcelona, Spain\\
\small Email: {\tt coniglio@unicamp.br}  \ and \ {\tt $\{$esteva, tommaso, godo$\}$@iiia.csic.es}
\date{}
}

\maketitle

\begin{abstract}
The aim of the paper is to analyze the expressive power of the square operator of {\L}ukasiewicz logic: $\ast x=x\odot x$, where $\odot$ is the strong 
{\L}ukasiewicz conjunction. In particular, we aim at understanding and characterizing those cases in which the square operator is enough to construct a finite MV-chain from a finite totally ordered set endowed with an involutive negation. The first of our main results shows that, indeed, the whole structure of MV-chain can be reconstructed from the involution and the \luk\ square if and only if the obtained structure has only trivial subalgebras and, equivalently, if and only if the cardinality of the starting chain is 
of the form $n+1$ where $n$ belongs to a class of prime numbers that we fully characterise. Secondly, we axiomatize the algebraizable matrix logic whose semantics is given by the variety generated by a finite totally ordered set endowed with an involutive negation and \luk's square operator.
Finally, we propose an alternative way to account for \luk\ square operator on involutive G\"odel chains. 
In this setting, we show that such an operator can be captured by a rather intuitive set of equations. 
\end{abstract}

\tableofcontents

\newpage



\section{Introduction}

The framework of the so-called {\em mathematical fuzzy logic} (MFL) encompasses a number of deductive systems conceived for reasoning with vague (in the sense of  gradual) information with a notion of comparative truth, and so formulas are usually interpreted in linearly ordered scales of truth values, which intend to represent gradual aspects of vagueness (or fuzziness). For a comprehensive and up-to-date account of MFL see the three volumes handbook \cite{Cintula-Hajek-Noguera:Handbook}. Two  interesting families of logics belonging to the family of MFL systems are given by the {\L}ukasiewicz hierarchy of $n$-valued logics \L$_n$ together with the infinite-valued version $\L$, on the one hand, and the G\"odel $n$-valued logics $G_n$, together with the infinite version $G$, on the other. 

The semantics of MFL systems follows, in general, the paradigm of (full) {\em truth-preservation}, according to which a formula is a consequence of a set of premises if every algebraic valuation that interprets the premises as absolutely true (value 1) also interprets the conclusion as absolutely true (value 1). It was observed  (see \cite{Font:DegreesSeriously}) that the {\em degree-preservation} paradigm (see \cite{Font-GTV:LogicPreservingDegrees,Bou-EFGGTV:PreservingDegreesResiduated}), according to which a formula follows from a set of premises if, for every evaluation, the truth degree of the conclusion is not lower than those of the premises, is more coherent with the many-valued approach to fuzzy logic. Indeed, within the degree-preserving consequence relations all the truth-values play an equally important role. 
As an intermediate alternative, it is possible to consider matrix logics in which the designated truth-values are given by (products of) order filters, see for instance ~\cite{CEGG} and~\cite{CEGG:Godel} for the case of (products of) {\L}ukasiewicz logics or  G\"odel's logics (possibly expanded with an involution) respectively.

%

Concerning {\L}ukasiewicz logics, it is well-known that $\L$ is algebraizable in the sense of Blok-Pigozzi (\cite{BP89}), having the variety $\mathbb{MV}$ of all MV-algebras as its equivalent quasivariety semantics, which is generated by the real interval $[0,1]$ equiped with suitable MV-operators, which is denoted by  $[0,1]_{MV}$. Algebraizability is preserved by finitary extensions, hence each finite-valued $\L$ukasiewicz logic $\L_{n+1}$ is also Blok-Pigozzi algebraizable by means of the subvariety  $\mathbb{MV}_{n+1}$ of MV-algebras generated by the standard $(n+1)$-valued {\L}ukasiewicz chain $\bL_{n+1}$ with domain $\{0, 1/n,\ldots, (n-1)/n, 1\}$. 
By means of a general result concerning equivalences between logics, based on  translations presented in~\cite{BP97}, the logic $\mathsf{L}_n^i$ characterized by the logical matrix $\langle \bL_{n+1}, F_{i/n}\rangle$ (where $F_{i/n}$ is the order filter generated by $i/n$) is also algebraizable by means of the variety $\mathbb{MV}_{n+1}$, see \cite{CEGG}\footnote{We warn the reader that the notation used in the present paper and that of \cite{CEGG} are not exactly the same. Indeed, while in \cite{CEGG} the MV-chain with $n+1$ elements is called {\em  MV$_n$-chain}, here, as we already used above, that chain will be called {\em MV$_{n+1}$-chain}. The same variation applies when we will speak about varieties generated by chains with $n+1$ elements.}.

Hilbert calculi characterizing the logics $\L_{n+1}$ are well known  (see, for instance, ~\cite{Cintula-Hajek-Noguera:Handbook}). By  a general result on equivalence between logics introduced in~\cite{BP97}, a sound and complete axiomatization can be obtained for each   logic $\mathsf{L}_n^i$ by translating the axioms and rules of a Hilbert calculus for $\L_{n+1}=\mathsf{L}_n^n$. However, the original signature of $\L_{n+1}$ does not result to be very natural for axiomatizing $\mathsf{L}_n^i$: in these logics the \luk\ implication is no longer an implication by any reasonable criteria, whenever $i < n$. 

Because of this, in~\cite{CEGG} we proposed an axiomatization of  $\mathsf{L}_3^1$ and $\mathsf{L}_3^2$ in terms of another signature $\Sigma$ given by $\neg$, $\vee$ and $*$, where $*$ represents the square (w.r.t. the strong {\L}ukasiewicz conjunction $\odot$) in  $\bL_{3+1}$, namely $*x:= x \odot x$. It turns out that in this signature it is possible to define the `classical' negation $\sim_{i/3}$ of the filter $F_{i/3}$ (for $i=1,2$):  ${\sim}_{i/3}x=0$ if $x \geq i/3$, and ${\sim}_{i/3}x=1$ otherwise. In turn, this induces  a `classical' (deductive) implication $x \to_{i/3} y := {\sim}_{i/3}x \vee y$, obtaining in this way a suitable and very natural language for axiomatizing the logics $\mathsf{L}_3^i$, for $i=1,2$. 


Despite the success in axiomatizing $\mathsf{L}_3^1$ and $\mathsf{L}_3^2$ in the signature $\Sigma = \{\vee, \neg, *\}$, it was observed  in~\cite{CEGG} that the issue of obtaining a `natural' axiomatization defined over such signature  for every $\mathsf{L}_n^i$ with $n > 3$ is a problem which ``appears to be much more complicated, and certainly it lies outside the scope of this paper'' (\cite[p. 150]{CEGG}). A crucial feature for the case $n=3$ mentioned  in~\cite{CEGG} is that \luk\ implication can be recovered from such signature. This feature does not hold for any $n$, not even for any prime number, as it is the case e.g.\ of $n=17$, as we shall see in Section \ref{sec:Lstaralgebras}. 
From this observation, a second question was posed in~\cite[p. 153]{CEGG}: the algebraic study of the fragment of $\bL_{n+1}$ defined in the signature $\Sigma = \{\vee, \neg, *\}$. These two questions stated in~\cite{CEGG}, namely, the formal study -- from the algebraic point of view -- of the implication-less reduct of the $(n+1)$-valued {\L}ukasiewicz chain $\bL_{n+1}$ expanded with the square operator $*$ (which will be denoted here by $\bL^*_{n+1}$), as well as the associated matrix logics $\Lambda_{n+1,i}^*=\langle \bL_{n+1}^*, F_{i/n}\rangle$ for every filter $F_{i/n}$ of designated values, constitute the starting point of the present paper.\footnote{To be precise, in~\cite{CEGG} both questions were posed only with respect to $n$ prime. This was motivated by the fact that $\mathsf{L}_n^i$, for $n$ prime and $i/n \leq 1/2$, constitute an interesting family of paraconsistent logics.}
In this manner, the present study, already initiated in two preliminary extended abstracts~\cite{CEFG} and~\cite{CEFG2}, will encompass both questions and more. 

Note that the square operator $\ast$ in the logics $\L_{n+1}$, or in $\L$, can be interpreted as a truth-stresser operator, in the sense of the class of truth-hedge operators axiomatically introduced by H\'ajek in~\cite{Ha01c} in the context of H\'ajek's Basic Fuzzy Logic BL to formalize the notion of `very true'. In fact, $*$ is a model of H\'ajek's truth-stresser operators for both \luk's and G\"odel fuzzy logics, as well as of the operators considered in a more general logical in the setting of MFL studied in \cite{EsGoNo:13}.

With respect to expressiveness, it is firstly proved in Section \ref{sec:Lstaralgebras} that, for $n \neq 4$, 
$\bL^*_{n+1}$ can define \luk\ implication (in other words, $\bL^*_{n+1}$  is term-equivalent to $\bL_{n+1}$) iff it is stricly simple, that is, it
has no non-trivial proper subalgebras. Surprisingly, and in contrast with the case of finite \luk\ chains, it will be shown that this does not hold true for all $n$ prime. Indeed,  for any prime number  $n \geq 3$, $\bL^*_{n+1}$ is term-equivalent to $\bL_{n+1}$ if and only if $n$ satisfies certain aritmetic property (see Theorem~\ref{mainThm}). For instance, $\bL^*_{n+1}$ cannot define \luk\ implication whenever $n> 5$ is a Fermat prime number (that is, $n$ is a prime of the form $n = 2^{2^m} + 1$ for some $m > 1$) such as $n=17$, $n=257$ or $n=65537$.\footnote{As of 2020, these are the only known Fermat primes greater than $5$.} On the other hand, any $\bL^*_{n+1}$ ($n$ being prime or not) can always define the order implication ($x \Rightarrow_c y =1$ if $x \leq y$, and  $x \Rightarrow_c y =0$ otherwise), and G\"odel implication $\Rightarrow_G$. This is an important fact from the point of view of the agebraic study of these structures, as we shall see.  

Concerning axiomatizations, it is proved in Section \ref{sec:logics}  that all the  matrix logics $\Lambda_{n+1,i}^*=\langle \bL_{n+1}^*, F_{i/n}\rangle$ are Blok-Pigozzi algebraizable with the same quasivariety over the signature $\Sigma=\{\vee,*, \neg\}$. Then, an uniform axiomatization for all of these logics it is obtained. The definition of these Hilbert calculi, together with the results on (un)characterizability of  $\bL_{n+1}$ in terms of $\Sigma$, constitute a complete solution of (an extended version of) the first problem posed in~\cite{CEGG}.
Concerning the algebraic study of these structures --the second question posed in~\cite{CEGG}-- it is also proved in Section \ref{sec:logics} that the variety generated by $\bL_{n+1}^*$ is constituted by $(n+1)$-valued G\"odel algebras with  involution  expanded by an unary operator $\star$ satisfying certain equations. This means that this class of algebras  can be axiomatized by means of equations, thus being a variety. 

Since not every subalgebra of $\bL_{n+1}^*$ is isomorphic to $\bL_{m+1}^*$ for some $m \leq n$, the question of studying  the behaviour of the square operator in subalgebras of $\bL_{n+1}^*$ is also tackled in the first part of Section \ref{sec:representable}. Let $[0,1]_{MV}^*$ be the algebra defined over the real unit interval by the \luk\ operations $\vee, \neg, \ast$. Since every $\bL_{n+1}^*$ is a (finite) subalgebra of $[0,1]_{MV}^*$, such study is realized by analyzing the finite subalgebras of this algebra.

As observed above, every  $\bL_{n+1}^*$ can define the G\"odel implication $\Rightarrow_G$; however, this operator (as well as the Monteiro-Baaz $\Delta$ operator) is not definable in $[0,1]_{MV}^*$. This suggests the  definition of a more comprehensive class of algebras, obtained by adding a unary $*$-like operator to G\"odel chains with an involutive negation, as it is done in the second part of Section \ref{sec:representable}. Finally,  the  G\"odel algebras with involutive negation and a $\star$ operation such that its implication free-reducts coincide with subalgebras of $\bL_{n+1}^*$ are axiomatically characterized.

Finally, let us mention that the structure of the paper is completed with some needed preliminaries gathered in the next section and with Section \ref{sec:conclusions} containing some conclusions and open problems.



\section{Preliminaries}\label{sec :preliminaries}
Along this paper we will be mainly concerned with the classes of finite chains belonging to the varieties $\mathbb{MV}$ of MV-algebras and $\mathbb{G}$ of G\"odel algebras. One of the most relevant class of algebras that contains both MV and G\"odel-algebras is the variety $\mathbb{BL}$ of H\'ajek's BL-algebras \cite{Ha98}. Let us start recalling that a BL-algebra is a bounded, integral and commutative residuated lattice ${\bf A}=(A, \land,\lor, \odot, \Rightarrow, 0,1)$ that further satisfies the following equations:
\begin{itemize}
	\item[-] $(x\Rightarrow y)\vee (y\Rightarrow x)=1$ \hfill ({\em prelinearity})
	\item[-] $x\wedge y=x\odot(x\To y)$ \hfill ({\em divisibility})
\end{itemize}
In every BL-algebra ${\bf A}$ one can define further operations. In particular, for all $a\in A$, the {\em residual negation} (or simply the {\em negation}) of $a$ is denoted by $\neg a$ and  stands for $a\Rightarrow 0$; also, for all  $a, b \in A$, $a\Leftrightarrow b$ is an abbreviation for $(a\Rightarrow b)\wedge(b\Rightarrow a)$. 

Further, a partial order relation $\leq$ can be defined: for all $a, b \in A$
\begin{center}
	$a\leq b$ iff $a\Rightarrow b=1$ holds.
\end{center}
The partial order $\leq$ coincides with the lattice order of ${\bf A}$. The BL-algebra ${\bf A}$ is said to be a {\em BL-chain} if $\leq$ is linear. 
\begin{definition}
	A BL-algebra ${\bf A}$ is said to be 
	\begin{itemize}
		\item[-] An {\em MV-algebra} if the equation $\neg\neg x=x$ holds in ${\bf A}$;
		\item[-] A {\em G\"odel-algebra} (or simply a G-algebra) if $x\odot y =x\wedge y$ holds in ${\bf A}$.
	\end{itemize}
	A BL-algebra, MV-algebra or G-algebra, is said to be {\em finite} if its universe is a finite set.
\end{definition}

It is worth to point out that finite MV and Godel chains are the ``building blocks'' of finite BL-chains. Indeed \cite[Corollary 3.7]{DBLP:journals/soco/CignoliEGT00} shows that finite BL-chains can only be ordinal sums of MV-chains and G-chains. One of the basic properties that distinguishes finite MV-chains from finite G-chains lies in the fact that, while MV-operations allow to describe the arithmetic sum between real numbers, in G\"odel chains is only possible to describe the order of their elements. 

In the rest of this paper, in order to ease the reading, we will distinguish MV-operations from G\"odel operations adopting subscripts: in particular, the implication operator of MV-algebras (also called \luk\ implication) will be denoted by $\Rightarrow_{\textrm \L}$, while G\"odel implication will be written $\Rightarrow_G$. The negation operators are defined as usual: MV-negation (or \luk\ negation) $\neg_{\textrm \L}x= x\Rightarrow_{\textrm \L}0$ and G\"odel negation $\neg_Gx=x\Rightarrow_G0$. 

The main differences between MV-algebras and G\"odel algebras can be easily grasped recalling how their operations behave in the standard algebras of the relative varieties. Recall in fact that both the variety $\mathbb{MV}$ and $\mathbb{G}$ can be generated by structures based on the real unit interval $[0,1]$. Those algebras, called respectively the {\em standard} MV-algebra (written $[0,1]_{MV}$) and the {\em standard} G\"odel algebra (denoted by $[0,1]_G$) interpret operations as follows: for all $x,y\in [0,1]$,
\begin{itemize}
	\item $x\odot y=\max\{0, x+y-1\}$; $x\wedge y=\min\{x,y\}$;
	\item $x\Rightarrow_{\textrm \L}y=\min\{1, 1-x+y\}$; $x\Rightarrow_G y=1$ if $x\leq y$ and $x\Rightarrow_G y=y$ otherwise;
	\item $\neg_{\textrm \L}x=1-x$; $\neg_G x=1$ if $x=0$ and $\neg_Gx=0$ otherwise.
\end{itemize}

In addition to the ones recalled above, in every MV-algebra, one can define further arithmetic operations like the bounded sum $x\oplus y=\neg_{\textrm \L} x\Rightarrow_{\textrm \L} y$ whose semantics in $[0,1]_{MV}$ is $x\oplus y=\min\{1, x+y\}$ and the square operator $\ast x=x\odot x$ that will play a main role in this paper and whose behavior in $[0,1]_{MV}$ is $\ast x=\max\{0, 2x-1\}$. 

Finite MV-chains are easily characterized. Indeed, for each natural number $n$, the set $\L_{n+1}=\{0, 1/n, 2/n, \ldots, (n-1)/n, 1\}$ is the domain of the $(n+1)$-valued MV-chain. Such algebra will be henceforth denoted by $\bL_{n+1}$. 
The G\"odel chain with $n+1$ elements will be denoted by ${\bf G}_{n+1}$.

Every finite MV-chain $\bL_{n+1}$   and every finite G\"odel chain ${\bf G}_{n+1}$ generate, respectively, proper subvarieties of $\mathbb{MV}$ and $\mathbb{G}$. Equations describing these subvarieties, within $\mathbb{MV}$ and $\mathbb{G}$, can be found e.g. in \cite{Grigolia} (for the case of MV) and \cite{Got01} (for the G\"odel case).

Notice that, by definition, \luk\ negation is involutive and thanks to this, all operations of any MV-algebra can be defined starting only from the signature 
$\{\Rightarrow_{\textrm \L}, 0\}$. 
In fact, we will use that reduced signature when we will deal with MV-algebras in the remaining of the present paper.

In turn, G\"odel negation $\neg_G$ does not satisfy the involutive equation $\neg\neg x=x$. For this reason, an expansion of G\"odel algebras by an involution has been studied in \cite{EGHN00} (see also \cite{FM06}). The corresponding algebraic structures are defined as follows.

\begin{definition}\label{def:IG}
	A {\em G\"odel algebra with involution} (IG-algebra for short) is a pair $({\bf A}, \msim)$ where ${\bf A}$ is a G\"odel algebra and ${\msim}:A\to A$ is a unary operator satisfying the following equations: 
	\begin{enumerate}
		\item $\msim\msim x=x$
		\item $\neg_G x\leq \msim x$
		\item $\Delta(x\To_G y)=\Delta(\msim y\To_G\msim x)$
		\item $\Delta x \vee \neg_G\Delta x=1$
		\item $\Delta(x\vee y)\leq \Delta x\vee \Delta y$
		\item $\Delta (x\To_G y) \leq \Delta x \To_G \Delta y$
	\end{enumerate}
	where $\Delta x=\neg_G\msim x$. 
\end{definition}
The class of $IG$-algebras form a variety that will be denoted by $\mathbb{IG}$. As it is proved in \cite[Theorem 7]{EGHN00}, $\mathbb{IG}$ is generated by the IG-algebra $([0,1]_G, \msim)$ where $[0,1]_G$ is the standard G-algebra and $\msim x=1-x$. The variety generated by the IG-chain with $n+1$ elements will be henceforth denoted by $\mathbb{IG}_{n+1}$.

It is worth noticing that the operator $\Delta$ appearing in Definition \ref{def:IG}, and that is definable in IG-algebras by combining the two negations $\neg_G$ and $\msim$, is the Baaz-Monteiro operator \cite{Baaz}. In every totally ordered algebra, $\Delta$ behaves as follows: $\Delta(x)=1$ if $x=1$ and $\Delta(x)=0$ otherwise. Such a operator, is indeed also definable in every finite MV-chain $\bL_{n+1}$ by the term $\Delta(x)=x^n=x\odot\ldots\odot x$ ($n$-times). Indeed, for every $0\leq k<n$, $(k/n)^n=0$ while $1^n=1$.  However, $\Delta$ is not definable in infinite MV-chains, whence in particular, it is not definable in $[0,1]_{MV}$.



\section{Analyzing the square operator: first steps} \label{sec:Lstaralgebras}

In this section we start the study on the expressive power of {\L}ukasiewicz's square operator $*$ by means of $(n+1)$-valued algebraic structures denoted by $\bL_{n+1}^*$.  After defining a fundamental algorithmic tool which allows to compute the subalgebras of $\bL_{n+1}^*$, we will analyse the relationship between primality of $n$ and term-equivalence between $\bL_{n+1}^*$ and $\bL_{n+1}$.

The next definition introduces the algebraic structures that will play a key role in this paper.  Let $\bL_{n+1}=(\L_{n+1}, \Rightarrow_{\textrm \L}, \neg_{\textrm \L}, 0,1)$ be the MV-chain with $n+1$ elements on the domain $\L_{n+1} = \{0, 1/n,\ldots, (n-1)/n, 1\}$ of $\bL_{n+1}$, where the strong conjunction $\odot$ is definable as usual, i.e.\ $x \odot y = \neg_{\textrm \L}(x \Rightarrow_{\textrm \L} \neg_{\textrm \L} y)$.

\begin{definition}
	The algebra $\bL_{n+1}^* = (\L_{n+1}, \vee, *, \neg_{\textrm \L},  0,1)$ is the structure obtained by adding the unary square operator $*:x\mapsto x\odot x$ and the join $\vee$ to the $\{\Rightarrow_{\textrm \L}\}$-free reduct of $\bL_{n+1}$.
\end{definition}
Therefore, for every $n$, $\bL^*_{n+1}$ is the linearly-ordered algebra on the domain $\{0, 1/n,\ldots, (n-1)/n, 1\}$ endowed with the operations $x \vee y= \max \{x,y\}$, $\neg_{\textrm \L} x=1-x$ and $$*x=\max\{0, 2x-1\},$$ besides the constants $0$ and $1$.  In every $\bL^*_{n+1}$-chain we can define the operation $+$ that is dual operation of $*$ w.r.t.\ to the negation $\neg$: $$+x=\neg_{\textrm \L}{*} \neg_{\textrm \L} x = \min\{1, 2x\}.$$  Furthermore, for every natural $n\geq 1$, we will denote by $*^n$ the $n$ times iteration of $*$, that is ${*}^{n}x=* x$ if $n=1$ and ${*}^{n}x={*}^{n-1}(* x)$ for $n>1$. This gives ${*}^{n}x= \max\{0, 2^n x -(2^n-1)\}$. 
Similarly, we define  ${+}^{n}x=+ x$ if $n=1$ and ${+}^{n}x={+}^{n-1}(+ x)$ otherwise, yielding ${+}^{n}x= \min\{1, 2^n x\}$. 

Recall that an element $x \in \L_{n+1}$ is called {\em positive} if $x > \neg_{\textrm \L} x$, i.e. $x > 1/2$, otherwise it is called {\em negative}. 

\subsection{On the subalgebras of $\bL^*_{n+1}$}

For every subset $X$ of $\L_{n+1}$ we will denote by $\langle X\rangle^*$ the subalgebra of $\bL_{n+1}^*$ generated by $X$. In case $X=\{x\}$ we will write $\langle x\rangle^*$ instead of $\langle\{x\}\rangle^*$. 

In the rest of this section we will only deal with \luk\ operations. Thus, in order to ease the reading, we will omit the subscript \L\ from operations without danger of confusion.

Let us present now an algorithmic tool that will be central  in the rest of this paper.

\begin{definition}[The procedure ${\tt P}$]\label{remProcP} Let us consider a procedure, that we will henceforth denote by ${\tt P}$, defined as follows: given $n$ and an element $a\in \L_{n+1}\setminus\{0,1\}$, ${\tt P}(n,a)$ iteratively computes a sequence $[a_1,\ldots, a_m,\ldots]$ of elements of $\L_{n+1}$ such that $a_1=a$ and for all $i\geq 1$,
	$$
	a_{i+1}=\left\{
	\begin{array}{ll}
	*a_i&\mbox{ if }a_i>1/2,\\
	\neg a_i&\mbox{ otherwise.}
	\end{array}
	\right.
	$$
	We say that ${\tt P}(n,a)$ {\em stops at $k$} (or {\em ${\tt P}(n,a)$ stops at $a_k$}) if $k$ is the first $i$ such that $a_{i+1}=a_j$ for some $j<i$.
	Since $\L_{n+1}$ is finite then, for every $a\in \L_{n+1}\setminus\{0,1\}$, there exists $k \geq 1$ such that ${\tt P}(n,a)$ stops at $k$. If ${\tt P}(n,a)$ stops at $k$, the {\em sequence generated by ${\tt P}(n,a)$} is denoted also by  ${\tt P}(n,a)=[a_1,\ldots, a_k]$, while the  {\em image} and the {\em negated image} of ${\tt P}(n,a)$ are ${\tt I}(n,a)=\{a_1,\ldots, a_k\}$ and ${\tt NI}(n,a)=\{\neg a_1,\ldots, \neg a_k\}$, respectively. The {\em range} of ${\tt P}(n,a)$ is ${\tt R}(n,a)={\tt I}(n,a) \ \cup \ {\tt NI}(n,a)$.
\end{definition}


Observe that, for every 
$a \in \Lns\setminus\{0,1\}$,  the set of positive elements of  ${\tt R}(n,a)$ coincides with the set of positive elements of ${\tt I}(n,a)$, i.e. the set ${\tt NI}(n,a)$ does not introduce new positive elements, i.e. all positive elements of ${\tt R}(n,a)$ belong to the sequence ${\tt P}(n,a)$ obtained by the procedure $P$ starting at $a$.

As a first application of the procedure ${\tt P}$, it will be shown that it allows us to compute the subalgebras of $\bL^*_{n+1}$ of the form  $\langle a\rangle^*$.

\begin{proposition} \label{subalgprin}
	Let $a \in \Lns\setminus\{0,1\}$. Then, the domain of the subalgebra $\langle a\rangle^*$ of $\bL^*_{n+1}$ is ${\tt R}(n,a) \cup \{0,1\}$.
\end{proposition}
\begin{proof}
	By definition of the procedure $\tt P$ and the above observation it is easy to prove that ${\tt R}(n,a) \cup \{0,1\}$ is closed by $\ast$ and $\neg$ and so it is the domain of a subalgebra containing $a$, i.e.  $\langle a\rangle^* \subseteq {\tt R}(n,a) \cup \{0,1\}$. Moreover every element of ${\tt R}(n,a)$ is obtained from $a$ using only the operations $\ast$ and $\neg$, hence ${\tt R}(n,a) \cup \{0,1\} \subseteq \langle a\rangle^*$. Therefore ${\tt R}(n,a) \cup \{0,1\} = \langle a\rangle^*$.
\end{proof}


Notice that if $a$ and $b$ are positive elements, and $b$ is reached by the procedurre ${\tt P}(n,a)$, i.e.\ $b \in  {\tt I}(n,a)$, then the sequence generated by ${\tt P}(n,b)$ is in fact a subsequence of the one generated by ${\tt P}(n,a)$, and hence, $\langle b \rangle^* \subseteq \langle a\rangle^*$. On the other hand, 
it is clear that for each $a \in \Lns\setminus\{0,1\}$, $\langle a\rangle^*= \langle \neg a\rangle^*$. Therefore from now on we will consider only subalgebras generated by positive elements since these subalgebras covers all one generated subalgebras. \vspace{0.2cm}

\noindent {\underline{\em Notation}}: In what follows, for every $\bL^*_{n+1}$-algebra, we will denote by $\co$ its coatom $(n-1)/n$

\begin{lemma}\label{lemmaNew1}
	For every $\bL^*_{n+1}$-algebra, if $\langle \co \rangle^*=\bL^*_{n+1}$, then every positive element $a$ of $\L^*_{n+1} \setminus \{1\}$ is reached by the procedure $\tt P$ starting at $\co$, that is, $a \in {\tt I}(n, \co)$.
\end{lemma}

\begin{proof} 
	It follows from the fact that the set of positive elements of ${\tt R}(n,\co)$ (which, by hypothesis and by Proposition \ref{subalgprin} is the set of positive elements of $\bL^*_{n+1} \setminus \{1\}$)  is included in ${\tt I}(n,\co)$.
\end{proof}

In the rest of this paper we will make often  use of the notion of {\em stricly simple} algebra whose definition is recalled below adapting to our case the general definition that can be found in \cite{GJKO}.

\begin{definition}\label{def:strictlySimple}
	An algebra $\bL^*_{n+1}$ is said to be {\em stricly simple} if its unique proper subalgebra is the two-element chain $\{0,1\}$.\footnote{The definition of {\em strictly simple} algebra usually requires an algebra ${\bf A}$ to have no non-trivial proper subalgebras, in other words, ${\bf A}$ is strictly simple if the trivial, one element algebra, is its unique subalgebra. However, in our context we always assume $0 \neq 1$, so that algebras have at least two elements, and thus the trivial algebra is the (Boolean) two-element algebra $\{0, 1\}$. 
	}
\end{definition}

Then we can prove the following.

\begin{lemma}\label{lemmaNew2}
	For every $\bL^*_{n+1}$-algebra, if $\langle \co\rangle^*=\bL^*_{n+1}$ and ${\tt P}(n,\co)=[a_1,\ldots, a_k]$ with $a_k=1/n$, then $\bL^*_{n+1}$ is strictly simple.
\end{lemma}

\begin{proof}
	Let $\langle \co \rangle^*=\bL^*_{n+1}$ and let $b$ be a positive element of $\bL^*_{n+1}$. By Lemma \ref{lemmaNew1}, $b \in {\tt P}(n,\co)$ and so the initial segment of ${\tt P}(n,b)$ is a subsequence of ${\tt P}(n,\co)$. In particular, the last element $1/n$ of ${\tt P}(n,\co)$ belongs to ${\tt I}(n,b)$ and so $c \in {\tt R}(n,b)$. 
	In other words $\langle b\rangle^*=\langle \co\rangle^*=\bL^*_{n+1}$, and the latter does not contain proper subalgebras and hence it is  strictly simple.
\end{proof}

Finally, we have the following characterization for strictly simple $\bL^*_{n+1}$-algebras.

\begin{theorem}\label{Thm:New1}
	For all $n> 1$ and $n\neq 4$, $\bL^*_{n+1}$ is strictly simple iff $\langle \co\rangle^*=\bL^*_{n+1}$. 
\end{theorem}

\begin{proof}
	The left-to-right direction is obvious.
	
	Let us hence assume that $\langle \co \rangle^*=\bL^*_{n+1}$. We distinguish the following cases:
	\begin{itemize}
		\item[-] $n$ is even: The case $n = 2$ clearly fulfils the claim. Then notice that for  any even number $n > 2$, $\langle 1/2 \rangle^*$ is a proper subalgebra of $\bL^*_{n+1}$, hence $\bL^*_{n+1}$ is not strictly simple. Thus the case $n = 4$ does not fulfill the claim since $\langle 3/4 \rangle^* = \bL^*_{5}$. 
		Now suppose $n > 4$. In order to get the claim, we have hence to prove that $\langle \co\rangle^* \neq \bL^*_{n+1}$. Since $n$ is even, 
		it is easy to see that every application of either $*$ or $\neg$ to a rational number with even denumerator, will output a rational with the same denominator and even numerator. In other words $\langle \co\rangle^* \setminus\{\co, \neg \co\}$ only contains rationals with even numerators, hence $\langle \co\rangle^*$ is a proper subalgebra of $\bL^*_{n+1}$.
		
		\item[-] $n > 1$ and odd: Let $a=((n+1)/2)/n$ be the least positive element of $\bL^*_{n+1}$ and let ${\tt P}(n,\co)=[a_1,\ldots, a_k]$. By Lemma \ref{lemmaNew1}, $a = a_t$ for some $a_t \in {\tt I}(n,\co)$. A direct computation shows that $a_{t+1} = *a=1/n$,  hence it must be $a_{t+1} = a_k$. 
		Thus, by Lemma \ref{lemmaNew2}, $\bL^*_{n+1}$ is strictly simple. 
	\end{itemize}
\end{proof}

Let us end this subsection with a comparison between $\bL_{n+1}$ and $\bL_{n+1}^*$-algebras concerning subalgebras and strictly simple algebras. In particular, recall that a finite MV-chain $\bL_{n+1}$ is strictly simple iff $n$ is prime \cite{Grigolia}. 

\begin{proposition}\label{subalebras} The following holds for every $n,m \geq 2$:
	\begin{itemize}
		\item If $\bL_{n+1}$ is subalgebra of $\bL_{m+1}$, then $\bL_{n+1}^*$ is subalgebra of $\bL_{m+1}^*$.
		\item If $\bL_{n+1}^*$ is stricly simple, then $n$ is prime.
	\end{itemize}
\end{proposition}
\begin{proof}
	The first item is a consequence of the fact that the  operation $\ast$ of $\bL_{n+1}^*$ is definable in $\bL_{n+1}$. The second item is a consequence of the first item plus the already recalled fact from \cite{Grigolia} stating that $\bL_{n+1}$ is strictly simple if and only if $n$ is prime.
\end{proof}

Notice that for every $\bL_{n+1}$, the subalgebras are algebras of type $\bL_{m+1}$ with $m$ being a divisor of $n$ and $\bL_{n+1}$ is strictly simple if $n$ is a prime number. Both statements are not true for $\bL_{n+1}^*$. There exists subalgebras of $\bL_{n+1}^*$ that are not of type $\bL_{m+1}^*$ and there exists prime numbers $n$ such that $\bL_{n+1}^*$ is not strictly simple as the following examples show.

\begin{example} \label{exemp1} 
	(1) Let $n=9$. Then, $\langle 8/9 \rangle^*=\{0,1/9, 2/9, 4/9, 5/9, 7/9, 8/9,1\}$. This subalgebra is a chain of 8 elements which is not isomorphic to $\bL_{7+1}^*$. Indeed,  in $\bL_{7+1}^*$, we have $\ast (5/7) = 3/7$, and the correspondent (w.r.t. the order) element  of $5/7$ in $\langle 8/9 \rangle^*$ is $7/9$. But in this algebra $\ast (7/9) = 2/9$, that corresponds to $2/7$ in $\bL_{7+1}^*$. This shows that, although both $\langle 8/9 \rangle^*$ and $\bL_{7+1}^*$ are  8-element algebras,  they are not isomorphic. \vspace{.1cm}
	
	(2) Let $n = 17$ that is prime and consider the subalgebra of $\bL^*_{18}$ generated by its co-atom $16/17$. A direct computation shows that
	$$\bigg\langle \frac{16}{17} \bigg\rangle^*=\bigg\{0,\frac{1}{17},\frac{2}{17},\frac{4}{17},\frac{8}{17},\frac{9}{17},\frac{13}{17},\frac{15}{17},\frac{16}{17},1\bigg\},$$
	which is in fact a proper non-trivial subalgebra of $\bL^*_{18}$ showing that the latter is not strictly simple.
\end{example} 
A more detailed study on the subalgebras of $\bL^*_{n+1}$-algebras will be the object of Subsection \ref{subsec:subLestrella1}.

\subsection{Term-equivalence between $\L_{n+1}^*$ and $\L_{n+1}$}\label{sec:termEquivalence}

In this section we will characterize those algebras $\bL^*_{n+1}$ that allow to define \luk\ implication $\Rightarrow_{\textrm \L}$ and hence, that are term-equivalent to the original finite MV-chain $\bL_{n+1}$. 

Let us start proving that in every $\bL^*_{n+1}$ we can define terms characterizing the principal order filter $F_a=\{b\in \L_{n+1}\mid b\geq a\}$ generated by $a$.

\begin{proposition}\label{prop3} For each  $a \in \L_{n+1}$, the unary operation $\Delta_a$ defined as 
	$$
	\Delta_a(x)=
	\left\{
	\begin{array}{ll}
	1&\mbox{ if } x\in F_{a}\\
	0 & \mbox{ otherwise.}
	\end{array}
	\right.
	$$
	is definable in  $\bL^*_{n+1}$. As a consequence, for every $a \in \L_{n+1}$, the operation $\chi_{a}$ that corresponds to the characteristic function of $a$ (i.e. $\chi_a (x)=1$ if $x = a$ and $\chi_a (x)=0$ otherwise) is definable as well.
\end{proposition}
\begin{proof} The case $a= 1$ corresponds to the Monteiro-Baaz $\Delta$ operator and, as is well-known, it can be defined as $\Delta_1(x) = {*}^n x$. 
	For the case $a=0$, define $\Delta_0(x)=\Delta_1(x) \vee \neg\Delta_1(x)$; this gives $\Delta_0(x)=1$ for every $x$. 
	
	In order to define $\Delta_a(x)$ for  $0< a  < 1$, consider the following notions.
	
	Given $a,b \in \L_{n+1}$ such that $a>b$ we say that $(a,b)$ is {\em separated} if either: (1) $a > 1/2 \geq b$, or (2) $b=0$, or (3) $a=1$. Clearly, if $(a,b)$ is not separated then either $1 > a > b > 1/2$ or  $1/2 \geq a > b> 0$.
	
	From now on we will consider terms $t(x)$ formed by combining applications of $*$ and $+$. Such terms are  monotonic, i.e., if $a \geq b$ then $t(a) \geq t(b)$. Observe that, for any $0<a<1$ there exists $m$ and $k$ such that ${+}^m a=1$ and ${*}^k a=0$. 
	%
	\begin{fact}
		If $(a,b)$ is separated then there exists a term $t(x)$ as above such that $t(a)=1$ and $t(b)=0$.
	\end{fact}
	Indeed, if (1) holds then $*a>*b=0$. Let $t(x)={+}^k*x$ such that $k= \min\{ m  \mid {+}^m*a=1 \}$. If (2) holds let $t(x)={+}^k x$ such that $k= \min\{ m  \mid {+}^m a=1 \}$.  If (3) holds let $t(x)={*}^k x$ such that $k= \min\{ m  \mid {*}^m b=0 \}$. In any case, $t$ is as required, by observing that $*1=1$ and $+0=0$.
	
	Now, given $0< a  < 1$, let $a^-$ be its immediate predecessor in the chain, i.e. $a^- = a - 1/n$. If $(a,a^-)$ is separated then, by {\bf Fact 1}, $\Delta_a(x)=t(x)$ is as required, since $t$ is monotonic. If $(a,a^-)$ is not separated, a sequence of pairs $(x_i,y_i)$ of elements in $\L_{n+1}$ such that $x_i> y_i$ will be defined by taking $x_0=a$, $y_0=a^-$ and by considering, for every $i\geq 0$,   the following two cases:\\[2mm]
	{\bf Case A:} Let $(x_i,y_i)$ such that $1 > x_i > y_i > 1/2$. Let $k_i= \max\{ m  \mid {*}^m x_i > {*}^m y_i \}$. Let $x_{i+1}=t_{i+1}(x_i)$ and  $y_{i+1}=t_{i+1}(y_i)$ where $t_{i+1}(x)= {*}^{k_i}x$. Note that $x_{i+1} > y_{i+1}$. If $(x_{i+1},y_{i+1})$ is separated the procedure stops. Otherwise, note that $1/2 \geq x_{i+1} > y_{i+1}> 0$. Go to {\bf Case B} with input $(x_{i+1},y_{i+1})$.\\[2mm]
	{\bf Case B:} Let $(x_i,y_i)$ such that $1/2 \geq x_{i} > y_{i}> 0$. If $x_i=1/2$ let $x_{i+1}={+}x_i$ and  $y_{i+1}={+}y_i$. Then, $(x_{i+1},y_{i+1})$ is separated and the procedure stops. Otherwise, if  $1/2 > x_{i}$, let $k_i= \max\{ m  \mid {+}^m x_i > {+}^m y_i \}$. Let $x_{i+1}=t_{i+1}(x_i)$ and  $y_{i+1}=t_{i+1}(y_i)$ where $t_{i+1}(x)= {+}^{k_i}x$.  Note that $x_{i+1} > y_{i+1}$.  If $(x_{i+1},y_{i+1})$ is separated the procedure stops. Otherwise, note that $1 >  x_{i+1} > y_{i+1} > 1/2$. Go to {\bf Case A} with input $(x_{i+1},y_{i+1})$. \vspace{2mm}
	
	By definition, if the procedure defined above stops then the output $(x_{i+1},y_{i+1})$ is separated. Note that $x_{i+1}=\bar {t}(a)$ while $y_{i+1}=\bar {t}(a^-)$ for some term $\bar {t}(x)$. In such a case, $\Delta_a(x)$ can be defined by the term $t(\bar {t}(x))$, where $t(x)$ is a term as specified  in {\bf Fact 1}. Indeed,  $\Delta_a(a)=t(x_{i+1})=1$ and $\Delta_a(a^-)=t(y_{i+1})=0$ and so $\Delta_a(x)$ is as required, and being a term constructed by combining applications of $*$ and $+$, it is monotonic.  Thus, it remains to prove that the procedure above always stops. But this is easy to see from the following observation: if $(a, b)$ is not separated then  either $(\# a, \# b)$ is separated or $\# a - \# b= 2(a-b)$, for $\# \in \{+,*\}$. This means that, either $(x_{i+1},y_{i+1})$ is separated, or  the distance $x_{i+1} - y_{i+1}$  between $x_{i+1}$ and $y_{i+1}$ is strictly greater than the distance $x_{i} - y_{i}$  between $x_{i}$ and $y_{i}$. Given that the distance between two elements $a$ and $b$ of $\L_{n+1}$ is itself elements of $\L_{n+1}$ (which is a finite set) and it is defined as $\neg(a\Rightarrow b)\vee\neg(b\Rightarrow a)$, a separated $(x_{i+1},y_{i+1})$ must be found  at some  step $i+1$.
	
	From the previous constructions we have shown that $\Delta_a(x)$ can always be constructed by means of a term which combines applications of $*$ and $+$. Finally, as for the operations $\chi_a$, define $\chi_1 = \Delta_1$; $\chi_0(x)= \neg \Delta_{1/n}(x)$, and  if $0 < a < 1$, then define $\chi_a(x) =  \Delta_a(x) \land \neg \Delta_{a^+}(x)$, where $a^+ = a + 1/n$ is the immediate sucessor of $a$. 
\end{proof}

Next we show an example of the above procedure to find the unary operations $\Delta_a$. 

\begin{example}
	Let us consider the $\bL_{n+1}^*$-chain for $n = 11$, and let $a = 8/11$. We show how we can find the operation $\Delta_{8/11}$ according to the procedure described in the proof of the previous proposition. In this case $a$ and $a^- = 7/11$ are both positive, so it fits with Case A above.
	Hence, the procedure above produces the following sequence of pairs (by simplicity, the denominator 11 will be omitted): $(8,7) \stackrel{*}{\mapsto}(5,3) \stackrel{+}{\mapsto} (10,6) \stackrel{*^2}{\mapsto} (7,0)$. Since $+(7/11)=1$, we obtain that $\Delta_{8/11}(x) = t(t_3(t_2(t_1(x)))) = {+}{*}^2{+}{*}x$, by using the notation of the proof of Proposition~\ref{prop3}. Similarly, one can check that 
	$\Delta_{9/11}(x) = t(t_1(x))$, where $t_1(x) = {*}^2 x$ and $t(x) = {+}^2 x$, i.e. $\Delta_{9/11}(x) = {+}^2{*}^2x$.  In this case, the pairs produced by the procedure described in the proof above are  $(9,8) \stackrel{*^2}{\mapsto}(3,0)$. Therefore, $\chi_{8/11}(x) = \Delta_{8/11}(x) \land \neg \Delta_{9/11}(x) = \min({+}{*}^2{+}{*}x, 1-  {+}^2{*}^2x)$.

	Table \ref{table1} shows, besides the operations $*$ and $+$, the different steps to obtain $\Delta_{8/11}$ and $\Delta_{9/11}$. The reader can easily obtain the other operators $\Delta_a$ from such table by applying the given  procedure.

	\begin{table}
		$${\small
			\begin{tabular}{| c | c | c |  c | c | c | c | c | c | c | c | c | c |}
			\hline
			$x$ & 0 & $\frac{1}{11}$& $\frac{2}{11}$& $\frac{3}{11}$& $\frac{4}{11}$& $\frac{5}{11}$& $\frac{6}{11}$& $\frac{7}{11}$ & $\frac{\bf 8}{\bf 11}$&  $\frac{\bf 9}{\bf 11}$& $\frac{10}{11}$& 1 \\
			\hline \hline
			$*x$ & 0 & 0 & 0 & 0 & 0 & 0 & $\frac{1}{11}$& $\frac{3}{11}$& $\frac{5}{11}$& $\frac{7}{11}$& $\frac{9}{11}$& 1 \\
			\hline \hline
			$+x$ & 0 &$\frac{2}{11}$& $\frac{4}{11}$& $\frac{6}{11}$& $\frac{8}{11}$& $\frac{10}{11}$&1& 1& 1& 1& 1 & 1 \\
			\hline \hline
			${+}{*}x$ & 0 & 0 & 0 & 0 & 0 & 0 & $\frac{2}{11}$& $\frac{6}{11}$& $\frac{10}{11}$& 1& 1 & 1 \\
			\hline
			${*}^2{+}{*} x$ & 0 & 0 & 0 & 0 & 0 & 0 & 0 & 0 & $\frac{7}{11}$& 1& 1 & 1 \\
			\hline
			${+}{*}^2{+}{*} x$ & 0 & 0 & 0 & 0 & 0 & 0 & 0 & 0 & 1 & 1& 1 & 1 \\
			\hline \hline
			${*}^2 x$ & 0 & 0 & 0 & 0 & 0 & 0 & 0 & 0 & 0 & $\frac{3}{11}$ & 1 & 1 \\
			\hline
			${+}^2{*}^2 x$ & 0 & 0 & 0 & 0 & 0 & 0 & 0 & 0 & 0 & 1 & 1 & 1 \\
			\hline
			\end{tabular}}
		$$
		\caption{Some definable operations in {\bf \L}$^*_{12}$} \label{table1}
	\end{table}
	
\end{example}

Actually, Proposition \ref{prop3} can be straightforwardly generalized to any subalgebra of a $\bL_{n+1}^*$.

\begin{corollary}\label{Delta}  Let $\bf A$ be a subalgebra of $\bL_{n+1}^*$. Then, for any element $a \in \bf A$,  the operations $\Delta_a$ and $\chi_a$ are also definable in $\bf A$. 
\end{corollary}

\begin{proof} Indeed, the same procedure defined in the proof of  Proposition \ref{prop3} to find the terms for $\Delta_a$ and $\chi_a$ in $\bL_{n+1}^*$ works in $\bf A$ as well, as the operations $*$ and $\neg$ are obviously closed in $\bf A$.  The argument given in that proof to show that the procedure always stops remains the same. 
\end{proof}

It is now almost immediate to check that the crisp (or order) implication as well as the G\"odel implication are definable in every $\bL^*_{n+1}$.

\begin{proposition} \label{godimp} The order implication and G\"odel implication,
	\begin{center}
		$x\Rightarrow_c y=
		\left\{\begin{array}{ll}
		1 &\mbox{ if }x\leq y\\
		0 & \mbox{ otherwise} 
		\end{array}
		\right.
		$\hspace{2cm}
		$x\Rightarrow_G y=
		\left\{\begin{array}{ll}
		1 &\mbox{ if }x\leq y\\
		y & \mbox{ otherwise} 
		\end{array}
		\right.
		$
	\end{center}
	are both definable in $\bL^*_{n+1}$.
\end{proposition}
\begin{proof}
	Indeed, $\Rightarrow_c$ can be defined as 
	\begin{equation}\label{eqImplication}
	x \Rightarrow_c y =\bigvee_{0\leq i \leq n}(\chi_{i/n}(x)\wedge \Delta_{i/n}(y)) . 
	\end{equation}
	In turn, G\"odel implication is given by 
	$
	x \Rightarrow_G y = (x\Rightarrow_c y)\vee y.
	$ 
\end{proof}

Now we are ready to  prove the main result of this section, that is a characterization of those algebras $\bL^*_{n+1}$ that define \luk\ implication $\Rightarrow_{\textrm \L}$ or, equivalently, of those algebras $\bL^*_{n+1}$ that are term-equivalent to $\bL_{n+1}$. First, recall how strictly simple algebras are introduced in Definition \ref{def:strictlySimple}.

\begin{theorem}\label{thm:termEquivalence}
	For all $n\neq 4$, the finite MV-chain $\bL_{n+1}$ is term equivalent to $\bL^*_{n+1}$ iff $\bL^*_{n+1}$ is strictly simple. 
\end{theorem}
\begin{proof}
	Left-to-right: If $\bL_{n+1}$ is term-equivalent to $\bL^*_{n+1}$ then \luk\ product $\odot$ is definable in $\bL^*_{n+1}$, and hence $\langle (n-1)/n \rangle^*=\bL^*_{n+1}$. Indeed, we can obtain $(n-i-1)/n = ((n-1)/n) \odot ((n-i)/n)$ for $i=1, \ldots, n-1$,  and $1 = \neg 0$. By Theorem~\ref{Thm:New1} it follows that $\bL^*_{n+1}$ is strictly simple.
	\vspace{.1cm}
	
	Right-to-left: since $\bL^*_{n+1}$ is strictly simple then,  for each $a,b \in \Ln$ where $a \notin\{0,1\}$ there is a definable term ${t}_{a,b}(x)$ such that ${t}_{a,b}(a) = b$. Otherwise, if for some $a\notin\{0,1\}$ and $b  \in \Ln$ there is no such term then ${\bf A}=\langle a\rangle^*$ would be a proper subalgebra of $\bL^*_{n+1}$  (since $b \not \in {\bf A}$) different from $\{0,1\}$, a contradiction.  
	%
	Now, for $0\leq j < i \leq n$ consider terms ${\bf t}_{i,j}(x,y)$ such that ${\bf t}_{i,j}(i/n,j/n)=(i/n) \Rightarrow_{\textrm \L}(j/n)$. Such terms can be defined as follows: if $n > i > j \geq 0$ then  ${\bf t}_{i,j}(x, y)= t_{i/n, a_{ij}}(x)$, where $a_{ij} = 1-i/n+j/n$; and ${\bf t}_{n,j}(x,y)= y$ for $0\leq j < n$. Since by Proposition~\ref{prop3} the operations $\chi_a(x)$ are definable for each $a \in \L_{n+1}$, then in  $\bL^*_{n+1}$ we can define {\L}ukasiewicz implication $\Rightarrow_{\textrm \L}$ as follows:  
	$$x\Rightarrow_{\textrm \L} y= (x \Rightarrow_c y)  \lor  \left (\bigvee_{n \geq i > j \geq 0} \chi_{i/n}(x)\wedge \chi_{j/n}(y) \land {\bf t}_{i,j}(x,y) \right) $$
	where $x\Rightarrow_c y$ is defined as in Proposition~\ref{godimp}.
\end{proof}


\begin{remark}
	The case $n=4$ is a singular one: it is the only counterexample for 
	Theorems~\ref{Thm:New1} and \ref{thm:termEquivalence}. First,  $\bL^*_{5}$ is generated by its coatom: $\bL^*_{5}=\langle 3/4\rangle^*$. In addition, it is term-equivalent to $\bL_{5}$. Indeed,  {\L}ukasiewicz implication $\Rightarrow_{\textrm \L}$ can be defined in $\bL^*_{5}$ as in the proof of  Theorem~\ref{thm:termEquivalence}, with suitable adaptations. For $0\leq j < i \leq 4$ consider terms ${\bf t}_{i,j}(x,y)$ such that ${\bf t}_{i,j}(i/4,j/4)=(i/4) \Rightarrow_{\textrm \L}(j/4)$. Such terms can be defined as follows (observing that $1/2=2/4$ in $\bL^*_{5}$): ${\bf t}_{4,j}(x,y)= y$ for $0\leq j < 4$; ${\bf t}_{i,0}(x,y)=\neg x$ for $0<i\leq 4$; ${\bf t}_{3,2}(x,y)=x$; ${\bf t}_{3,1}(x,y)=*x$; and ${\bf t}_{2,1}(x,y)=\neg y$. Then, $x\Rightarrow_{\textrm \L} y$ is given by the term
	$$x\Rightarrow_{\textrm \L} y= (x \Rightarrow_c y)  \lor  \left (\bigvee_{4 \geq i > j \geq 0} \chi_{i/4}(x)\wedge \chi_{j/4}(y) \land {\bf t}_{i, j}(x,y) \right).$$
	However, $\bL^*_{5}$ is not strictly simple, since it has the non-trivial subalgebra with domain $\{0,1/2,1\}$.
\end{remark}

%
%

\subsection{Strictly simple $\bL^*_{n+1}$-chains and prime numbers}\label{subsec:primes}
As we have shown in Proposition \ref{subalebras}, $n$ being prime is a necessary condition for $\bL^*_{n+1}$ to be strictly simple which, in turn, is equivalent to the term-equivalence between $\bL_{n+1}$ and $\bL^*_{n+1}$ (if $n\neq 4)$ by Theorem \ref{thm:termEquivalence}. However, the primality of $n$ is not a sufficient condition for $\bL^*_{n+1}$ to be strictly simple. In fact, as the following result shows, there are prime numbers $n$ for which  $\bL^*_{n+1}$ contains non-trivial subalgebras. This fact was already observed in Example \ref{exemp1} (2).

\begin{lemma}\label{cor:Fermat}
	If $n > 5$ is of the form $n = 2^m+1$ then $\bL_{n+1}$ and $\bL^*_{n+1}$ are not term-equivalent. 
\end{lemma}
\begin{proof}
	Let $n$ be of the form $n = 2^m +1$ for some $m>2$.
	If  $\co = (n-1)/n$ then $*^{m}\co=1/n$. By Proposition 
	\ref{subalgprin}, the algebra $\langle \co\rangle^*$ has domain ${\tt I}(n,\co) \ \cup \ {\tt NI}(n,\co) \cup \{0,1\}$ (recall Definition~\ref{remProcP}). Since ${\tt I}(n,\co)$ has $m+1$ elements then ${\tt NI}(n,\co)\setminus {\tt I}(n,\co)$ has at most $m-1$ elements (since $\co=\neg (1/n)$ and  $1/n=\neg \co$ belong to  ${\tt NI}(n,\co)\cap {\tt I}(n,\co)$). Hence, the  algebra $\langle \co\rangle^*$ has at most $2 + 2(m-1) + 2 = 2m+2$ elements. Since $2m+2< 2^m+1=n$  as $m>2$, $\langle \co\rangle^*$ is properly contained in $\bL_{n+1}$, and it is different from $\{0,1\}$. Therefore $\langle \co\rangle^*$ is a proper non trivial subalgebra of $\bL^*_{n+1}$, and the result follows from Theorem~\ref{thm:termEquivalence}. 
\end{proof}

It is well-known that if  $n = 2^m+1$  is prime then $m$ is of the form $2^k$; in such case, $n=2^{(2^k)}+1$ is said to be a {\em Fermat prime}. As mentioned in the introduction, up to 2020 the only known Fermat primes are  3, 5, 17, 257, and 65537, and it is an open problem to determine whether there are infinitely many such prime numbers. Therefore, for any prime  $n> 5$ of the form $n = 2^m+1$ (i.e., for any Fermat  prime $> 5$),
$\bL^*_{n+1}$ contains non-trivial subalgebras, and hence it is not strictly simple. 

Notice that, as we showed in Example \ref{exemp1} (2), the subalgebra of $\bL^*_{18}$ generated by its coatom is a proper subalgebra of $\bL^*_{18}$ and indeed, 17 is the first Fermat prime number greater than $5$.

%

We have seen that, in contrast with the case of the $\bL_{n+1}$-algebras, which are strictly simple iff $n$ is prime, there are prime numbers for which $\bL^*_{n+1}$ contains proper non-trivial subalgebras. It is however possible to characterize those prime numbers which ensure the  term-equivalence between $\bL_{n+1}$ and $\bL^*_{n+1}$. Let us start by the following definition.

\begin{definition}\label{def:campinas}
	Let $\Pi$ be the set of 	
	odd primes $n$ such that $2^m$ is not {congruent with $\pm 1$} mod $n$ for {all $m$ such that} $0 < m < (n-1)/2$.\footnote{These prime numbers are known in the literature as those {\em odd primes with one coach}. Properties satisfied by such a set of prime numbers can be found in the following webpage of the Online Encyclopedia of Integer Sequences (OEIS): https://oeis.org/A216371.
		Further interesting properties on the class $\Pi$ can be found in \cite{HP2010}}
\end{definition}

By  Fermat's little theorem, $2^{n-1}$ is congruent with 1 mod $n$, for every odd prime $n$. Since $2^{n-1}=b^2$ for $b=2^m$ and $m = (n-1)/2$, it follows that $b^2$ is congruent with 1 mod $n$. But then, using that $n$ is prime, we conclude that $b=2^m$ is congruent with $\pm 1$ mod $n$, for $m = (n-1)/2$. From this, $n$ is in  $\Pi$ iff $n$ is an odd prime such that $(n-1)/2$ is the least $m>0$ such that $2^m$ is congruent with $\pm 1$ mod $n$. 

As a matter of example, the first prime numbers (below 200) in the set $\Pi$ are:
3, 5, 7, 11, 13, 19, 23, 29, 37, 47, 53, 59, 61, 67, 71, 79, 83, 101, 103, 107, 131, 139, 149, 163, 167, 173, 179, 181, 191, 197 and 199.

The following Theorem \ref{mainThm} is the main result of this subsection and it characterizes the class of prime numbers for which the \luk\ implication is definable in $\bL^*_{n+1}$ (besides $n=2$).
Before proving it, we need to show the following lemma. 

\begin{lemma}\label{lemmaKey}
	For each odd number $n$, the procedure $\mathtt{P}$ starting at $\co = (n-1)/n$ stops after reaching $1/n$, that is, if $\mathtt{P}(n, \co) = [a_1, a_2, \ldots, a_t]$ then $a_t = 1/n$.
\end{lemma}

\begin{proof}
	We already observed that $\mathtt{P}$ always stops since $\Lns$ is finite. Thus assume, by way of contradiction, that $\mathtt{P}$  stops at $a_t=k/n$ with $k>1$. 
	
	\begin{fact}\label{fact0}
		Let $q$ and $n>q$ be positive integers and $n$ odd. Then: (1) if $q > n/2$,  $*(q/n)$ has always an odd numerator; (2) if $q$ is odd, then $\neg(q/n)$ has even numerator. 
	\end{fact}
	
	By Fact \ref{fact0}, $k$ cannot be even. Indeed, if $k$ were even, then $a_t$ would be obtained from $a_{t-1}$ by negating it, i.e.\  $a_t = \neg a_{t-1}$, and then $a_{t+1} = *(a_t)$ should coincide with a previous element $a_i$ in the list $[a_1, a_2, \ldots, a_t]$ such that $ a_i < a_1 = (n-1)/n$. But then $a_{t+1} = a_i$ should have an odd numerator, and hence $a_i = *(a_{i-1})$, that is, we would have  $a_{t+1} = *(a_t) = *(a_{i-1})$, hence $a_t = a_{i-1}$, and the procedure should have stopped at $a_{t-1}$, contradiction. Therefore $k$ must be odd. 
	
	Since $\mathtt{P}$ stops at $a_m = k/n$, 
	there exists an $a_j < a_1$ already met by the procedure such that $a_{t+1} = a_j$. If $a_{t+1} = *(a_t)$, then we reason as above and get $a_{j-1} = a_{t}$, contradiction. So let us assume  $a_{t+1} = \neg (a_t) = n - (k/n)=(n-k)/n=a_j$. 
	Notice that the numerator $n-k$ of $a_j$ is even. 
	Thus, $a_j$ must have been obtained as $\neg (a_{j-1})$ in a previous step, that is, $a_j = \neg a_{j-1}=\neg a_t= a_j$, and hence it must be the case that $a_{j-1}=a_t$. In other words, the procedure should have stopped earlier at $a_{t-1}$, contradiction. 
	Therefore, necessarily $k=1$, that is to say, $a_t=1/n$. 
\end{proof}

\begin{theorem}\label{mainThm}
	Let  $n \geq 3$ be an odd number. Then: $\bL_{n+1}$ and $\bL^*_{n+1}$ are term-equivalent iff $n$ is  a prime number belonging to the set $\Pi$.
\end{theorem}
\begin{proof}
	%
	Let  $\co = (n-1)/n$ and let $\mathtt{P}(n, \co) = [a_1, \ldots, a_l]$ be the sequence generated by the procedure $\mathtt{P}(n, \co)$. This sequence can be regarded as the concatenation of a number $r$ of subsequences 
	\begin{center}
		$[a^1_1, \ldots, a^1_{l_1}]$, $[a^2_1, \ldots, a^2_{l_2}]$, \ldots, $[a^r_1, \ldots, a^r_{l_r}]$, 
	\end{center}
	with $a^1_1 = a_1$ and $a^r_{l_r} = a_l$, where for each subsequence $1\leq j \leq r$, only the last element $a^j_{l_j}$ is below 1/2, while the rest of elements are above 1/2.  
	
	By the very definition of $*$, it follows that the last elements $a^j_{l_j}$ of every subsequence are of  the form
	\begin{center}
		$a^j_{l_j} =  
		\begin{cases} 
		\frac{k_j n-2^{m_j}}{n}, \mbox{ if $j$ is odd} \\[2mm] 
		\frac{2^{m_j}-k_j n}{n}, \mbox{ otherwise, i.e. if $j$ is even} 
		\end{cases}$\\[4mm]
	\end{center}
	for some $m_j,k_j>0$,  where in particular $m_j$ is the number of strictly positive elements of $\Ln$ which are obtained by the procedure before getting $a^j_{l_j}$.
	
	By Lemma \ref{lemmaKey}, since $n$ is odd, $\mathtt{P}(n, \co)$ stops at $1/n$, i.e. $a_l = a^r_{l_r} = 1/n$. 
	Thus, writing $m$ and $k$ istead of $m_r$ and $k_r$:
	\begin{center}
		$\begin{cases} kn-2^m = 1, \mbox{ if $r$ is odd (i.e., $2^m \equiv -1$ (mod $n$) if $r$ is odd)} \\[2mm] 
		2^m-kn = 1, \mbox{ otherwise (i.e., $2^m \equiv 1$ (mod $n$) if $r$ is even)}  \end{cases}$\\[2mm]
	\end{center}
	where $m$ is now the number of strictly positive elements in the list $\mathtt{P}(n, \co)$, i.e. that are reached by the procedure before stopping. Therefore, $2^m$ is congruent with $\pm 1$ mod $n$.
	
	Suppose now that $n\geq 3$ is an odd number such that $\bL_{n+1}$ is term equivalent to $\bL^*_{n+1}$. Then, $\bL^*_{n+1}$ is strictly simple and by Proposition~\ref{subalebras}, $n$ is  prime. This being so, the integer $m$ defined above must be exactly $(n-1)/2$, the number of strictly positive elements of  $\L_{n+1}$ (different from 1). Otherwise, $\langle \co\rangle^*$ would be a proper subalgebra of it, which is absurd. Moreover, for no $m'<m$ one has that $2^{m'}$ is congruent with $\pm 1$ mod $n$ because, in this case, the algorithm would stop producing, again, a proper subalgebra of $\bL^*_{n+1}$. This shows that $n \in \Pi$, i.e., the left-to-right direction of our claim.
	
	In order to show the other direction, assume that $\bL_{n+1}$ and $\bL^*_{n+1}$ are not term-equivalent. By Theorem~\ref{thm:termEquivalence}, this implies that  $\bL^*_{n+1}$ is not strictly simple. Thus, by Theorem~\ref{Thm:New1}, $\langle a_1\rangle^*$ is a proper subalgebra of $\bL^*_{n+1}$ and hence the algorithm above stops in $1/n$, after reaching $m<(n-1)/2$ strictly positive elements of $\Ln$. Thus, $2^m$ is congruent with $\pm1$ mod $n$ (depending on whether $r$ is even or odd, where $r$ is the number of subsequences in the sequence $\mathtt{P}(n, \co)$ as described above), showing that $n\not\in \Pi$.
\end{proof}

Observe that 3 and 5 are the only Fermat primes belonging to $\Pi$. Indeed, by Lemma~\ref{cor:Fermat}, if $n$ is a Fermat prime such that $n> 5$ then $\bL_{n+1}$ and $\bL^*_{n+1}$ are not term-equivalent. By Theorem~\ref{mainThm}, $n$ does not belong to the set $\Pi$.



\section{The matrix logics of $\L^*_{n+1}$-chains} \label{sec:logics}
Given the algebra $\bL_{n+1}^*$, it is possible to consider, for every $1\leq i\leq n$,  the matrix logic $\Lambda_{n+1,i}^*=\langle \bL_{n+1}^*, F_{i/n}\rangle$, where $F_{i/n} = \{ a \in \L_{n+1} \mid a \geq i/n\}$. 
Recall that the logic $\Lambda_{n+1,i}^*$, regarded as a consequence relation over a propositional language $\mathcal L$ with signature $(\vee,*, \neg, \bot)$ of type $(2,1,1,0)$, is defined as follows: for every subset of formulas $\Gamma \cup \{\varphi\} \subseteq {\mathcal L}$,\\

\begin{tabular}{lll}
	$\Gamma \models_{{\Lambda}_{n+1,i}^*} \varphi$ & if & for every  $\bL_{n+1}^*$-evaluation $e$, \\
	&& $e(\psi) \geq i/n$ for every $\psi \in \Gamma$ implies $e(\varphi) \geq i/n$
\end{tabular}
\mbox{} \\

In the following subsections we will first show that the logics $\Lambda_{n+1,i}^*$ are algebraizable, then we will describe their equivalent algebraic semantics, and finally we will provide an axiomatization. 

\subsection{Algebraizability of the logics $\Lambda_{n+1,i}^*$}

In this section we show that all the  logics $\Lambda_{n+1,i}^*$ are algebraizable in the sense of Blok-Pigozzi \cite{BP89}, and that, for every $i,j$, the quasi-varieties associated to $\Lambda_{n+1,i}^*$ and $\Lambda_{n+1,j}^*$ are  the same. 

%
%

Consider in  $\bL_{n+1}^*$ the strict equality operation $\approx$ defined as $x \approx y = 1$ if $x=y$, and $x \approx y =0$ otherwise. In fact, in $\bL_{n+1}^*$ the operation $\approx$ is definable by means of the implication $\Rightarrow_c$ (recall Proposition \ref{godimp}) 
as follows: 
$$
x \approx y = (x \Rightarrow_c y) \wedge (y \Rightarrow_c x).
$$ 
Equivalently, it can also be defined by means of the G\"odel implication $\Rightarrow_G$ and the unary operator $\Delta_1$ (Propositions \ref{prop3} and \ref{godimp}) as 
$$
x \approx y = \Delta_1((x \Rightarrow_G y) \wedge (y \Rightarrow_G x)).
$$
This fact allows us to show the algebraizability of the logics $\Lambda_{n+1,n}^*$. 
\begin{lemma}\label{lemma1}
	For every $n$, the logic $\Lambda_{n+1}^*:=\Lambda_{n+1,n}^*=\langle\bL_{n+1}^*, \{1\} \rangle$ is algebraizable. 
\end{lemma}
\begin{proof}
	Consider the set of formulas in two variables $\Theta(p,q)=\{p \approx q\}$ and the formulas in one variable $ \delta(p), \epsilon(p) $, where $\delta(p) = p$ and $\epsilon(p) = \Delta_0(p)$. Then, according to \cite[Th. 4.7]{BP89}, we have to show that the following conditions hold:
	\begin{itemize}
		\item[(i)]   $\models_{{\Lambda}_{n+1}^*} \Theta(\varphi, \varphi)$
		\item[(ii)] $ \Theta(\varphi, \psi) \models_{{\Lambda}_{n+1}^*} \Theta(\psi, \varphi)$
		\item[(iii)] $\Theta(\varphi, \psi), \Theta(\psi, \chi) \models_{{\Lambda}_{n+1}^*} \Theta(\varphi, \chi)$
		\item[(iv-a)]  $\Theta(\varphi, \psi) \models_{{\Lambda}_{n+1}^*} \Theta( \# \varphi, \# \psi)$, for $\# \in \{\neg, *\}$
		\item[(iv-b)]  $\Theta(\varphi, \psi), \Theta(\chi, \nu) \models_{{\Lambda}_{n+1}^*} \Theta( \varphi \lor \chi, \psi \lor \nu)$
		\item[(v)] $\varphi \models_{{\Lambda}_{n+1}^*} \Theta( \delta(\varphi), \epsilon(\varphi))$, and $\Theta( \delta(\varphi), \epsilon(\varphi)) \models_{{\Lambda}_{n+1}^*} \varphi$
	\end{itemize}
	Observe that for every $\bL_{n+1}^*$-evaluation $e$ and formulas $\varphi, \psi$, we have $e(\Theta(\varphi, \psi)) = 1$ iff $e(\varphi) = e(\psi)$, and since $e(\epsilon(\varphi)) = 1$, we also have $e(\Theta( \delta(\varphi), \epsilon(\varphi))) = 1$ iff $e(\varphi) = 1$. Then it is very easy to check that all the above conditions are satisfied. 
\end{proof}

Actually, the logic $\Lambda_{n+1}^*$ is a (Rasiowa) {\em implicative}  logic, since it satisfies the following characteristic properties (see e.g. \cite{CN2010}):
\begin{itemize}
	\item[(R1)] $\models_{\Lambda_{n+1}^*} \varphi \To \varphi$
	\item[(R2)] $\varphi \To \psi, \psi \To \chi \models_{\Lambda_{n+1}^*}  \varphi \To \chi$
	\item[(R3)] $\varphi \Leftrightarrow \psi  \models_{\Lambda_{n+1}^*}  \neg \varphi \To \neg \psi$, \quad 
	$\varphi \Leftrightarrow \psi \models_{\Lambda_{n+1}^*}  * \varphi \To * \psi$
	
	$\varphi_1 \Leftrightarrow \psi_1, \varphi_2 \Leftrightarrow \psi_2 \models_{\Lambda_{n+1}^*}  \varphi_1 \lor  \varphi_2  \To \psi_1 \lor \psi_2$
	
	\item[(R4)] $\varphi, \varphi \To \psi  \models_{\Lambda_{n+1}^*}  \psi$
	\item[(R5)] $ \varphi  \models_{\Lambda_{n+1}^*} \psi \To \varphi$
\end{itemize}
where $\To\; \in \{\To_c, \To_G\}$. And it is well-known that implicative logics are algebraizable, see  \cite[in Vol. 1 of \cite{Cintula-Hajek-Noguera:Handbook}]{Cintula-Noguera}. 

Blok and Pigozzi  introduce in \cite{BP97} the following notion of equivalent deductive systems. Two propositional deductive systems $S_{1}$ and $S_{2}$ in the same language are {\em equivalent} if there are translations $\tau_{i}:S_i \to S_j$ for $i \neq j$ such that: $\Gamma\vdash_{S_{i}}\varphi$ iff $\tau_{i}(\Gamma)\vdash_{S_{j}}\tau_{i}(\varphi)$, and $\varphi\dashv\vdash_{S_{i}}\tau_{j}(\tau_{i}(\varphi))$. 
From very general  results in \cite{BP97}, it follows that two equivalent logic systems  are indistinguishable from the algebraic point of view, namely: if one of the systems is algebraizable  then the other will be also algebraizable w.r.t. the same quasivariety. This  can be applied to 
$\Lambda_{n+1,i}^*$.

\begin{lemma}\label{lemma2}
	For  every $n$ and every $1\leq i\leq n-1$, the logics $\Lambda_{n+1}^*$ and $\Lambda_{n+1,i}^*$ are equivalent. 
\end{lemma}
\begin{proof}
	Indeed, it is enough to consider the translation mappings $\tau_1:\Lambda_{n+1}^* \to \Lambda_{n+1,i}^*$, $\tau_1(\varphi)=\Delta_1(\varphi)$, and  $\tau_{i,2}:\Lambda_{n+1,i}^* \to \Lambda_{n+1}^*$, $\tau_{i,2}(\varphi)=\Delta_{i/n}(\varphi)$. 
\end{proof}

Therefore, as a direct consequence of Lemma \ref{lemma1}, Lemma \ref{lemma2} and the above observations, the algebraizability of $\Lambda_{n+1,i}^*$ easily follows. 

\begin{theorem}\label{thm:alg}
	For every $n$ and every $1\leq i\leq n$, the logic $\Lambda_{n+1,i}^*$ is algebraizable. 
\end{theorem}
%

Therefore, for each logic $\Lambda_{n+1,i}^*$ there is a quasivariety $\mathbb{\Lambda}(i,n)$ which is its equivalent algebraic semantics. 
%
The question of describing $\mathbb{\Lambda}(i,n)$ is dealt with in the next section, where it is shown that it is in fact a variety.   


\subsection{The equivalent algebraic semantics of $\Lambda^*_{n+1,i}$}\label{subsec:EASLambda}


Due to Lemma \ref{lemma2}, all logics $\Lambda_{n+1, i}^*$'s are equivalent to  $\Lambda_{n+1}^*$ and hence they have the same equivalent algebraic semantics, i.e. $\mathbb{\Lambda}(i,n) = \mathbb{\Lambda}(j,n)$, for every $1\leq i,j \leq n$, and hence we will simplify the notation and refer to $\mathbb{\Lambda}(n)$ for this common quasi-variety. In order to characterize $\mathbb{\Lambda}(n)$, in the following we consider, without loss of generality, the case $i = n$, i.e. the algebras corresponding to the matrix logic $\Lambda_{n+1}^*= \langle \bL_{n+1}^*, F_{1}\rangle$ defined by the filter $F_1 = \{1\}$. 

We start by observing that from the chain  $\bL_{n+1}^* = (\L_{n+1}, \land, \lor, *, \neg, 0, 1)$, we can define the algebra
$${\bf IG}_{n+1} = (\L_{n+1}, \land, \lor, \To_G, \neg, 0, 1)$$ 
where $\To_G$ is G\"odel implication, that is definable as shown in Proposition \ref{godimp}, that is in fact the standard $(n+1)$-valued G\"odel algebra with an involution \cite{EGHN00,FM06}. Conversely, $\bL_{n+1}^*$ can be seen as the expansion of ${\bf IG}_{n+1}$ with the $*$ operation. Recalling also from Proposition \ref{prop3} the definition, for a given $n$ and for every $a \in \L_{n+1}$, of the terms $\Delta_a$ (as suitable sequences of the $\neg$ and $*$ operations), the above motivates the following definition.

\begin{definition}\label{def:GInvStarAlgebras}
	An  {\em $\Lambda^\star_{n+1}$-algebra} 
	is a triple $({\bf A}, \ninv, \star)$, where  
	\begin{itemize}
		\item ${\bf A} = (A, \land, \lor, \Rightarrow, 0, 1)$ is a $(n+1)$-valued G\"odel algebra (a $G_{n+1}$-algebra for short), 
		\item $({\bf A}, \ninv)$ is a $(n+1)$-valued G\"odel algebra with involution (a $IG_{n+1}$-algebra for short), and
		\item  $\star$ is a unary operation on $A$ such that the following equations hold, where for every $a \in \{0, 1/n, \ldots, (n-1)/n, 1\}$, the operation $\Delta_a$  is defined as a sequence of $\ninv$'s and $\star$'s obtained from its definition in Proposition \ref{prop3} by replacing the occurrences of $\neg$ and $*$ by $\ninv$ and $\star$ respectively: 
		\begin{itemize}
			\item[(Eq1)]   $\Delta_1 (x) = \Delta(x)$, $\Delta_0(x) = 1$
			\item[(Eq2)]   {$\Delta_a \Delta_b x = \Delta_b x$ } 
			\item[(Eq3)]   $\Delta_{a}x \vee \ninv\Delta_{a}x = 1$  
			\item[(Eq4)]    $\Delta_{a^+} x \Rightarrow \Delta_{a} x = 1$, \ \ if \ $a < 1$
			
			\item[(Eq5)]   { $  \Delta_a(x \vee y) = (\Delta_a x \vee \Delta_a y) $ }

			\item[(Eq6)]   {$ \Delta_{\neg a}\ninv x  \Rightarrow \ninv \Delta_{a^+} x = 1$,  \ \ if \ $a < 1$}

			
			\item[(Eq7)]   
			$\Delta (x \To y) \To (\star x \To \star y) = 1$
			
			\item[(Eq8)]    { $\Delta_a x \To \Delta_{{*}a} {\star}x = 1$}
			\item[(Eq9)]    { $\Delta_{({*}a)^+} {\star}x  \To \Delta_{a^+} x = 1$} 
			
		\end{itemize}
		where $\Delta(x) =  \ninv x \To 0$ is the Baaz-Monteiro operator, and $a^+ = a + 1/n$. 
	\end{itemize}
\end{definition}

Since the class of $IG_{n+1}$-algebras is a variety (it is a subvariety of the class of G\"odel algebras with an involution), from the above definition it is clear that the quasi-variety  $\mathbb{\Lambda}(n)$ of $\Lambda^\star_{n+1}$-algebras  is in fact a variety.

Moreover, defining $ x \Leftrightarrow y = (x \Rightarrow y) \land (y \Rightarrow x)$,  the following congruence law holds for $\star$: 
\begin{center}
	\hfill $ \mbox{if } x \Leftrightarrow y = 1 \mbox{ then } {\star}x \Leftrightarrow \star y = 1$ \hfill (Cong)
\end{center}
If we look at a $\Lambda^\star_{n+1}$-algebra as an axiomatic expansion of its underlying (prelinear) $IG_{n+1}$-algebra with the additional $\star$ operation, (Cong) is in fact the necessary condition to be satisfied by $\star$ to keep the prelinearity property in the expanded algebra, see e.g.  \cite[in Vol. 1 of \cite{Cintula-Hajek-Noguera:Handbook}]{Cintula-Noguera}. Therefore, the variety $\mathbb{\Lambda}(n)$ is semilinear and the following subdirect representation holds.

\begin{proposition}\label{LambdaSemilinear} Every $\Lambda^\star_{n+1}$-algebra is a subdirect product of linearly ordered $\Lambda^\star_{n+1}$-algebras. 
\end{proposition}

Looking at the above axioms, we observe that (Eq7) requires $\star$ to be a non-decreasing operation, while (Eq1) declares that the $n$-iteration of $\star$ results in the well-known Baaz-Monteiro's $\Delta$ operator. These two properties allows us to prove the following three further basic properties of the $\star$ operation. 

\begin{lemma} The following identities hold in any $\Lambda^\star_{n+1}$-algebra:
	\begin{itemize}
		\item[(i)]  $\star x \Rightarrow x = 1$
		\item[(ii)] $\star 1 = 1$
		\item[(iii)] $\star 0 = 0$
	\end{itemize}
\end{lemma}
\begin{proof}
	\item[(i)] By the above representation theorem, it is enough to prove it for linearly-ordered $\Lambda^\star_{n+1}$-algebras. Let $\bf A$ be a $\Lambda^\star_{n+1}$-chain, and by way of contradiction, let $x \in A$ such that $x < \star x$. By (Eq7) and (Eq1), we have the following chain of inequalities: $x < \star x \leq {\star\star}x \leq \ldots \leq (\star)^n x = \Delta_1(x) = \Delta(x)$. 
	But if $x < \star x$ it means that $x <1$ and hence $\Delta x = 0$. It then follows that $\star x = 0$, in contradiction with the hypothesis $x < \star x$. 
	
	\item[(ii)] By (Eq8), $1= \Delta_1 1 \To \Delta_{*1} {\star}1 = \Delta_1 1 \To \Delta_{1} {\star}1$, but by (Eq1), $\Delta = \Delta_1$ and we know that $\Delta 1 = 1$, thus $\Delta {\star}1 = 1$, and hence  ${\star}1 = 1$ as well. 
	\item[(iii)] It directly follows from (i) by taking $x = 0$. 
	%
	%

	
\end{proof}

Recall the operations  $\chi_a$'s definable from the $\Delta_a$'s as $\chi_a(x) = \Delta_a(x) \land \ninv \Delta_{a^+}(x)$ for $a < 1$ and  $\chi_1(x) = \Delta_1(x)$. Next lemma shows some properties of these operations. 

\begin{lemma} \label{impeq} The following equations hold in the variety of $\Lambda^\star_{n+1}$-algebras:
	\begin{itemize}
		\item[(i)] $\bigvee_{a \in \textrm{\L}_{n+1}} \chi_a x = 1$
		\item[(ii)] $\chi_a x \land \chi_b x = 0$, hence $\ninv (\chi_a x \land \chi_b x) = 1$, for $a \neq b$
		\item[(iii)] $\chi_a x = \chi_{\neg a} \ninv x$
		\item[(iv)] $\chi_a x \Rightarrow \chi_{\ast a}{\star}x = 1$ 
	\end{itemize}
	Moreover, in any $\Lambda^\star_{n+1}$-chain, the following monotonicity condition holds: 
	\begin{itemize}
		\item[(v)] If $x \leq y$ and $\chi_a(x) = \chi_b(y) = 1$ then $a \leq b$.
	\end{itemize}
\end{lemma}

\begin{proof}
	
	\begin{itemize}
		\item[(i)]  
		%
		By definition of the operators $\chi_a$, it is easy to check that $\bigvee_{0 \leq a \leq 1} \chi_{a} x = \Delta_1 x \,\vee\, \Delta_{(n-1)/n} x \,\vee\, \ldots \,\vee\, \Delta_{1/n} x \,\vee\, \ninv\Delta_{1/n} x$, but $\Delta_{1/n} x \,\vee\, \ninv\Delta_{1/n} x = 1$, hence $\bigvee_{0 \leq a \leq 1} \chi_{a} x = 1$ as well. 
		
		\item[(ii)] 
		W.l.o.g., suppose $a > b$. By definition, $\chi_a x \land \chi_b x = (\Delta_a x \land \ninv \Delta_{a^+} x) \land (\Delta_b x \land \ninv \Delta_{b^+} x)$. 
		Since $a > b$ then $a \geq b^+$ and so $\Delta_a x \leq \Delta_{b^+} x$  by (Eq4). Hence, $\chi_a x \land \chi_b x \leq \Delta_{b^+} \alpha \land \ninv \Delta_{b^+} \alpha = 0$. 
		\item[(iii)] 
		%
		If $a=0$  the result follows by (Eq6), namely: $\chi_0 x = \ninv \Delta_{1/n}x = \Delta_1 \ninv x = \chi_1 \ninv x$. If $a=1$ then $\chi_1 x= \chi_1 \ninv \ninv x = \chi_0 \ninv x$. 
		Now, suppose that $0 < a < 1$. Then,
		$\chi_a x  = \Delta_a x \land \ninv \Delta_{a^+} x$, and since $x=\ninv \ninv x$, $\Delta_{a} x = \ninv \Delta_{(\neg a)^+} \ninv x$. By (Eq6) again, $\ninv \Delta_{a^+} x = \Delta_{\neg a} \ninv x$. Therefore, $\chi_a x = \ninv \Delta_{(\neg a)^+} \ninv x \,\land\, \Delta_{\neg a} \ninv x = \chi_{\neg a} \ninv x$. 
		
		\item[(iv)] 
		Note first that $\Delta_{({*}a)^+} {\star}x  \leq \Delta_{a^+} x$ iff $\ninv \Delta_{a^+}x \leq \ninv \Delta_{({*}a)^+} {\star}x$.  Then, from (Eq8), (Eq9)  we get $\Delta_a x \land \ninv \Delta_{a^+}x \leq  \Delta_{{*}a} {\star}x \land \ninv \Delta_{({*}a)^+} {\star}x$, that is, $\chi_a x \leq \chi_{{*}a}{\star}x $. 
		
		\item[(v)] In a given $\Lambda^\star_{n+1}$-chain $\bf A$, the condition is equivalent to the following one: for all $x, y\in A$, if  $\chi_a(x) = \chi_b(x \lor y) = 1$ then $a \leq b$; and in turn this equivalent to: if  $\chi_a(x) = 1$ and $a > b$ then $\chi_b(x \lor y) = 0$. Now, by definition if  $\chi_a(x) = 1$ we have $\Delta_a x = 1$ and, by Equation (Eq5), $\Delta_a(x \lor y) = 1$ as well. Then, since $b^+ \leq a$, by Equation (Eq4),  we have $\Delta_{b^+}(x \lor y) = 1$, i.e., $\ninv \Delta_{b^+}(x \lor y) = 0$, and again by definition of $\chi_b$, we finally have $\chi_b(x \lor y) = 0$. 
		
	\end{itemize}
\end{proof}

\begin{lemma} Every $\Lambda^\star_{n+1}$-chain $({\bf A}, \ninv, \star)$ is isomorphic to a subalgebra of $\bL^*_{n+1}$. 
\end{lemma}

\begin{proof} Let $\bf A$ be a  $\Lambda^\star_{n+1}$-chain. Since in particular the G-reduct of $\bf A$ is a G$_{n+1}$-chain, $\bf A$ is finite, and let $| A | = m+1 \leq n+1$ and $A = \{0 < a_1 < \ldots < a_{m-1} < 1\}$. Note that, by the symmetry induced by the involutive negation, we have $\ninv a_j = a_{m-j}$. We will show that $\bf A$ embeds into the standard algebra $\bL^*_{n+1}$. 
	
	By (i) and (ii) of Lemma \ref{impeq}, for each $a_j \in A$, there is a unique $i_j \in \{0, 1, \dots, n\}$ such that $\chi_{i_j/n}(a_j) = 1$. Let us check that 
	$\bar{A} = \{0, i_1/n, \ldots, i_{m-1}/n, 1\}$ is the domain of a subalgebra of cardinality $m+1$ of $\bL^*_{n+1}$. It is clear that $\bar{A}$ is closed under the G\"odel operations $\land, \lor, \To$, thus we only have to check that ${\neg}(i_j/n), {*}(i_j/n) \in \bar{A}$, for each $i_j/n \in \bar{A}$: 
	
	\begin{itemize}
		\item[(i)] by (iii) of Lemma \ref{impeq},  if $\ninv a_j  = a_k$, then $1 = \chi_{i_j/n}(a_j) = \chi_{i_k/n}(a_k) =  \chi_{{\neg}(i_j/n)} (a_k)$, hence by (i) and (ii) of Lemma \ref{impeq}, ${\neg}(i_j/n) = i_k/n \in \bar{A}$.

		\item[(ii)] by (iv) of Lemma \ref{impeq}, if $\star a_j = a_k$, then $1 = \chi_{i_j/n}(a_j) = \chi_{i_k/n}(a_k) = \chi_{{*}(i_j/n)} (a_k)$, hence by (i) and (ii) of Lemma \ref{impeq}, ${*}(i_j/n) = i_k/n \in \bar{A}$.  
	\end{itemize}
	
	\noindent Note that, by the symmetry induced by the involutive negation, we have $n - i_j = i_{m-j}$ for every $j \in \{1, \ldots, m\}$. 
	Then we define a mapping $h: A \to \L^*_{n+1}$ by stipulating $h(0) = 0, h(1) = 1$ and $h(a_j) = i_j/n$ for all $j = 1, \ldots, m-1$. It is clear that $h$ is one-to-one and is order preserving (by (v) of Lemma \ref{impeq}), and hence a morphism w.r.t.\ G\"odel operations. Moreover, $h$ is a morphism w.r.t.\ to the $\ninv$ and $\star$ operations as well: 
	
	\begin{itemize}
		\item[-] $h(\ninv a_j) = h(a_{m-j}) = i_{(m-j)/n} = 1- i_j/n = \neg h(a_j)$
		
		\item[-] since ${*}(i_j/n) \in \bar{A}$, then let $i_k/n = {*}(i_j/n)$ and hence $\star a_j = a_k$. Then $h(\star a_j) = h(a_k) = i_k/n =  {*}(i_j/n) = {*}h(a_j)$
	\end{itemize}
	Therefore, $\bf A$ is isomorphic to the subalgebra of  $\bL^*_{n}$ over the domain $\bar{A} = \{0, i_1/n, \ldots, i_{m-1}/n, 1\}$. 
\end{proof}

As a consequence we have the following result. 

\begin{theorem}\label{thm:GenerateLambda}  The variety of $\Lambda^\star_{n+1}$-algebras is generated by the algebra $\bL^*_{n+1}$. 
\end{theorem}

The result above immediately shows  that the variety of $\Lambda_{n+1}^\star$-algebras is the equivalent algebraic semantics of the logic $\Lambda^\star_{n+1}$. Indeed, by definition, for every finite set of formulas $\Gamma\cup\{\varphi\}$, we have that $\Gamma \models_{{\Lambda}_{n+1}^\star} \varphi$ iff for every  $\bL_{n+1}^*$-evaluation $e$, $e(\psi) =1$ for every $\psi \in \Gamma$ implies $e(\varphi) =1$ iff, by Theorem \ref{thm:GenerateLambda}, for every  $\Lambda_{n+1}^\star$-algebra and every $\Lambda^\star_{n+1}$-evaluation $e$, $e(\psi) =1$ for every $\psi \in \Gamma$ implies $e(\varphi) =1$. This observation, together with Lemma~\ref{lemma2}, leads to the following result.

\begin{corollary} The variety of $\Lambda^\star_{n+1}$-algebras is the equivalent algebraic semantics of the logics $\Lambda^*_{n+1, i}$ for every $1\leq i\leq n$. 
\end{corollary}

\subsection{A uniform axiomatization of the logics $\Lambda^*_{n+1,i}$}
Now, we present a uniform axiomatization for the logics $\Lambda^*_{n+1,i}$. Let us remark that  the calculus we are going to present in this section provides and alternative axiomatization to the one that can be obtained by translating the algebraic equations defining the variety of $\Lambda^*_{n+1}$-algebras. 

Our signature $\Sigma$ contains two unary connectives ${\star}$ (square) and $\ninv$ (negation), plus a binary connective  $\vee$ for disjunction. In this signature, the following derived connectives will be useful: \\

\begin{tabular}{ll}
	- &  $\alpha \wedge \beta := \ninv(\ninv \alpha \vee \ninv \beta)$ \\
	- & $\Delta_a$, for each $a \in \L_{n+1}$, as defined in (the proof of) Proposition \ref{prop3}, \\ & replacing all the occurrences of $\neg$ and $*$ by $\ninv$ and $\star$ respectively\footnotemark \\
	- & $\chi_a \alpha := \Delta_a \alpha \land \ninv \Delta_{a^+} \alpha$, if $0 < a < 1$, where $a^+ = a +(1/n)$\\
	- & $\chi_0 \alpha := \ninv\Delta_{1/n} \alpha$;  $\chi_1 \alpha := \Delta_1 \alpha$\\
	- & $\alpha \Rightarrow_c \beta := \bigvee_{0\leq i \leq n}(\chi_{i/n}(x)\wedge \Delta_{i/n}(y))$ \\
	- & $\finv_{i/n} \alpha := \ninv \Delta_{i/n} \alpha$\\  
	- & $\alpha \rightarrow_{i/n} \beta := {\finv}_{i/n}\alpha \vee \beta =  \ninv \Delta_{i/n} \alpha \lor \beta$\\
	- &  $\alpha \leftrightarrow_{i/n} \beta := (\alpha \rightarrow_{i/n} \beta) \wedge (\beta \rightarrow_{i/n} \alpha)$  \\ \\
\end{tabular}
\footnotetext{Recall that, by definition, $\Delta_1 \alpha = (\star)^n \alpha$ and  $\Delta_0 \alpha =  (\star)^n \alpha \lor \ninv  (\star)^n \alpha$.}

In order to keep notation lighter, and without risk of confusion, the subscript $i/n$ will be omitted from the symbols 
$\rightarrow_{i/n}$ and $\leftrightarrow_{i/n}$.

\begin{definition}  \label{calLin} The Hilbert  calculus  ${\sf AX}^*_{n+1,i}$  for the logic $\Lambda_{n+1,i}^*$, defined over the signature $\Sigma$, is given as follows:\\[2mm]
	{\em Axiom schemas:} those of {\sf CPL} (propositional classical logic) restricted to the signature $(\vee, \rightarrow)$\footnote{Namely, the schemas
		$\alpha \to (\alpha \lor \beta)$, $\alpha \to (\beta \lor \alpha)$, $(\alpha \to \gamma) \to ((\beta \to \gamma) \to (\alpha \lor \beta \to \gamma)$, $\alpha \to (\beta \to \alpha)$, $(\alpha \to \beta) \to ((\beta \to \gamma) \to (\alpha \to \gamma))$, and $\alpha \vee (\alpha \rightarrow \beta)$.}
	plus the following ones, where $a,b \in \{0, 1/n, \ldots, (n-1)/n, 1\}$: 
	\begin{itemize}
		
		\item[(Ax1)]    $(\alpha \leftrightarrow \beta) \rightarrow (\ninv \alpha \leftrightarrow \ninv \beta)$
		\item[(Ax2)]    $\ninv\ninv\alpha \leftrightarrow \alpha$ 
		\item[(Ax3)]    $\ninv(\alpha \vee \beta) \rightarrow \ninv \alpha$ 
		\item[(Ax4)]    $\ninv \alpha \rightarrow (\ninv\beta \rightarrow \ninv(\alpha \vee \beta))$
		
		
		\item[(Ax5)]    {$\Delta_a \Delta_b \alpha \leftrightarrow \Delta_b \alpha$} 
		\item[(Ax6)]    $\Delta_{a}\alpha \vee \ninv\Delta_{a}\alpha$  
		\item[(Ax7)]    $\Delta_{a^+} \alpha \rightarrow \Delta_{a} \alpha$ 
		
		\item[(Ax8)]    $  \Delta_a(\alpha \vee \beta) \leftrightarrow (\Delta_a\alpha \vee \Delta_a\beta)$ 

		\item[(Ax9)]   $ \Delta_{\neg a}\ninv\alpha  \leftrightarrow \ninv \Delta_{a^+}\alpha$,  \ \ if \ $ a < 1$

		
		\item[(Ax10)]   {$\Delta_{i/n}\alpha \to \alpha $}
		
		\item[(Ax11)]   { $\Delta_a \alpha \to \Delta_{{*}a} {\star}\alpha$}
		\item[(Ax12)]   { $\Delta_{({*}a)^+} {\star}\alpha  \to \Delta_{a^+} \alpha$} 
	\end{itemize}
	\noindent {\em Inference rule:}
	\begin{itemize}
		\item[(MP)] $\displaystyle\frac{\alpha \ \ \ \ \alpha\rightarrow
			\beta}{\beta}$
	\end{itemize}
\end{definition}

It is easy to prove that the usual axioms involving $\land$ of positive classical propositional logic CPL$^+$, over $(\land, \lor, \to)$,  can be derived in the system ${\sf AX}^*_{n+1,i}$ by means of the axioms (Ax1)-(Ax4), thus the logic  $\Lambda_{n+1,i}^*$ in fact contains CPL$^+$. Moreover, 
it is worth noting that the system ${\sf AX}^*_{n+1,i}$ satisfies the deduction-detachment theorem w.r.t. the implication $\rightarrow$, namely: 
\begin{center}
	$\Gamma \cup\{ \alpha\} \vdash_{{\sf AX}^*_{n+1,i}} \beta$ iff $\Gamma \vdash_{{\sf AX}^*_{n+1,i}}  \alpha \rightarrow \beta$,
\end{center}
for every set of formulas $\Gamma \cup \{\alpha,\beta\}$. Indeed, it is well-known that any logic presented by means of a Hilbert calculus and containing a binary connective $\rightarrow$ such that  the schemas  \\

(A1):~$\alpha \rightarrow (\beta \rightarrow \alpha)$ 

(A2):~$(\alpha \rightarrow (\beta \rightarrow \gamma)) \rightarrow ((\alpha \rightarrow \beta) \rightarrow (\alpha \rightarrow \gamma))$  \\

\noindent 
are derivable, and where (MP) (w.r.t. $\rightarrow$) is the only inference rule, satisfies the deduction-detachment theorem w.r.t.  $\rightarrow$. 
In addition, ${\sf AX}^*_{n+1,i}$ satisfies the following metaproperty (sometimes called {\it proof by cases}): 
\begin{center}
	$\Gamma,\alpha \vdash_{{\sf AX}^*_{n+1,i}} \gamma$ and $\Gamma,\beta \vdash_{{\sf AX}^*_{n+1,i}} \gamma$ implies that  $\Gamma,\alpha \vee \beta \vdash_{{\sf AX}^*_{n+1,i}} \gamma$.
\end{center}
This is a consequence of the deduction-detachment theorem and {\sf CPL}. Besides, the conjunction $\land$ (defined as above) satisfies in this logic the classical properties, namely: $\alpha \to (\beta \to (\alpha \land \beta))$, $(\alpha \land \beta) \to \alpha$, and $(\alpha \land \beta) \to \beta$. This can be easily proven by using axioms (Ax1)-(Ax4) and MP.

Also, observe that Axiom (Ax6), together with items~(i) and~(x) in Lemma~\ref{theorems} below, capture the fact the $\Delta_a$'s connectives are Boolean in the sense that formulas built from expressions $\Delta_a \varphi$ with connectives $\lor, \ninv, \to$ behave as in classical logic, and thus one can classically reason with them. Formulas of this kind will be called {\em Boolean}. We will provide a formal justification of this statement a bit later.

Next lemma gathers some interesting theorems of ${\sf AX}^*_{n+1,i}$ that follow from the above axiomatics. 

\begin{lemma} \label{theorems} The following are theorems of ${\sf AX}^*_{n+1,i}$, where $a, b \in \{0, 1/n, \ldots, (n-1)/n, 1\}$: 
	
	\begin{itemize}

		\item[(i)]  {$\Delta_{a}\alpha \rightarrow (\ninv\Delta_{a}\alpha \rightarrow \beta)$} 
		
		\item[(ii)]    {$\Delta_{a^+}\alpha \rightarrow (\Delta_{\neg a}\ninv\alpha \rightarrow \beta)$} 
		\item[(iii)]   {$\Delta_{a^+}\alpha \vee \Delta_{\neg a}\ninv\alpha$ } 
		\item[(iv)]    {$\Delta_{a}\alpha \leftrightarrow \Delta_{a}\ninv\ninv\alpha$} 
		
		
		\item[(v)]   { $\alpha \to \Delta_{i/n}\alpha$}
		\item[(vi)] $(\alpha \land {\finv}_{i/n}\alpha) \to \beta$
		
		\item[(vii)]  {  $  \chi_a(\alpha \vee \beta) \rightarrow (\chi_a\alpha \vee \chi_a\beta)$} 
		\item[(viii)] $\bigvee_{a \in \textrm{\L}_{n+1}} \; \chi_a \alpha$
		\item[(ix)] { $(\chi_a \alpha \land \chi_b \alpha) \to \beta$, for $a \neq b$}
		\item[(x)] { $(\Delta_a \alpha \to \Delta_b \beta) \leftrightarrow (\ninv \Delta_b \beta \to \ninv \Delta_a \alpha)$}  
		\item[(xi)] $(\chi_a \alpha \land \chi_b \beta) \to \chi_{\max(a, b)} (\alpha \lor \beta)$
		\item[(xii)] $\chi_a \alpha \leftrightarrow \chi_{\neg a} \ninv \alpha$
		\item[(xiii)] $\chi_a \alpha \rightarrow \chi_{\ast a}{\star}\alpha$ 
		\item[(xiv)] $\chi_a \alpha \to \chi_{\Delta_b(a)} \Delta_b \alpha$ 
		\item[(xv)] $\Delta_{a} \alpha \leftrightarrow \lor_{b \geq a} \; \chi_b \alpha$
		\item[(xvi)] $\Delta_b \alpha \to \Delta_a \alpha$, if $b \geq a$
	\end{itemize}
\end{lemma}

\begin{proof} The proofs of all the cases are as follows.   
	\begin{itemize}
		
		\item[(i)] 
		By definition of $\to$, we have $\Delta_{a}\alpha \rightarrow (\ninv\Delta_{a}\alpha \rightarrow \beta) = \ninv\Delta_{i/n} \Delta_a \alpha \lor ( \ninv \ninv\Delta_{i/n} \Delta_a \alpha \lor \beta)$, and by applying (Ax5), (Ax1) and (Ax2) (as well as {\sf CPL}), the latter is equivalent to 
		$ (\ninv \Delta_a \alpha \lor  \Delta_a \alpha) \lor \beta$, which is clearly a theorem of ${\sf AX}^*_{n+1,i}$ by axiom (Ax6) and {\sf CPL}.



		\item[(iv)] 
		The case $a=0$ is obviously true, by definition of $\Delta_0$. Suppose now that $a > 0$. 
		From (Ax9), $\Delta_{\neg b}\ninv\alpha  \leftrightarrow \ninv \Delta_{b^+}\alpha$ is a theorem, for every $0 \leq b < 1$. By taking $b=a^- = a - 1/n$  we get  $ \Delta_{\neg (a^-)}\ninv\alpha  \leftrightarrow \ninv \Delta_{a}\alpha$, and so  $ \Delta_{a}\alpha \leftrightarrow \ninv \Delta_{\neg (a^-)}\ninv\alpha$, by (Ax1), (Ax2) and  {\sf CPL}. Noticing that $\neg(a^-) = (\neg a)^+$, $\ninv \Delta_{\neg (a^-)}\ninv\alpha$ is  $\ninv \Delta_{(\neg a)^+}\ninv\alpha$. By applying (Ax9) again to this last formula, and taking into account that $\neg\neg a=a$, we finally have the following chain of equivalences: 
		$\Delta_a \alpha \leftrightarrow  \ninv \Delta_{(\neg a)^+}\ninv\alpha  \leftrightarrow  \Delta_a\ninv \ninv \alpha$.

		\item[(v)]  
		It directly follows by definition of $\to$: $\alpha \to \Delta_{i/n}\alpha = \ninv \Delta_{i/n} \alpha \lor \Delta_{i/n} \alpha$, the latter being a theorem by (Ax6). 
		
		
		\item[(vi)] Notice that $\alpha \land {\finv}_{i/n}\alpha = \alpha \land \ninv \Delta_{i/n} \alpha$  and, due to~(v), this implies $\Delta_{i/n} \alpha \land \ninv \Delta_{i/n} \alpha$, which implies any $\beta$ by~(i).

		\item[(vii)] 
		If $a=1$ the result follows by~(Ax8). If $a=0$ then $\chi_a(\alpha \vee \beta) =  \ninv \Delta_{1/n}  (\alpha \vee \beta)$, which implies $\ninv (\Delta_{1/n} \alpha \vee \Delta_{1/n}\beta)$, by (Ax8), (Ax1) and  {\sf CPL}. The latter implies $\ninv \Delta_{1/n} \alpha$, by (Ax3), and this implies  $\ninv \Delta_{1/n} \alpha \vee \ninv \Delta_{1/n}\beta$, by  {\sf CPL}.
		Suppose now that $0< a < 1$. Then,
		$\chi_a(\alpha \vee \beta) =   \Delta_a(\alpha \vee \beta) \land \ninv \Delta_{a^+}  (\alpha \vee \beta)$ is equivalent to $(\Delta_a\alpha \vee \Delta_a\beta) \land \ninv (\Delta_{a^+}\alpha \vee \Delta_{a^+}\beta)$, by (Ax8) and (Ax1). The latter is equivalent to 	$(\Delta_a\alpha \vee \Delta_a\beta) \land  \ninv \Delta_{a^+}\alpha \land \ninv \Delta_{a^+}\beta$, by definition of $\land$ and (Ax1)-(Ax4). But this is equivalent to $(\Delta_a\alpha \land  \ninv \Delta_{a^+}\alpha \land \ninv \Delta_{a^+}\beta) \lor ( \Delta_a\beta \land  \ninv \Delta_{a^+}\alpha \land \ninv \Delta_{a^+}\beta)$, by  {\sf CPL}. By using {\sf CPL} again, this formula implies $(\Delta_a\alpha \land  \ninv \Delta_{a^+}\alpha) \lor ( \Delta_a\beta \land \ninv \Delta_{a^+}\beta)$, that is, $\chi_a\alpha \vee \chi_a\beta$. 
		
		\item[(viii)]  
		%
		By item (i) and {\sf CPL} it is easy to see that $\bigvee_{0 \leq a \leq 1} \chi_{a} \gamma$ is equivalent to $\Delta_1\gamma \,\vee\, \Delta_{(n-1)/n} \gamma \,\vee\, \ldots \,\vee\, \Delta_{1/n} \gamma \,\vee\, \ninv\Delta_{1/n} \gamma$, and the latter is a theorem of  ${\sf AX}^*_{n+1,i}$, by (Ax6) and the properties of $\vee$ coming from  {\sf CPL}.
		
		\item[(ix)] 
		W.l.o.g., suppose $a > b$. By definition, $\chi_a \alpha \land \chi_b \alpha = (\Delta_a \alpha \land \ninv \Delta_{a^+} \alpha) \land (\Delta_b \alpha \land \ninv \Delta_{b^+} \alpha)$. 
		Since $a > b$ then $a \geq b^+$ and so $\Delta_a \alpha \to \Delta_{b^+}\alpha$ is a theorem, by (Ax7) and {\sf CPL}. Hence, by  {\sf CPL}  once again, $\chi_a \alpha \land \chi_b \alpha$ implies $\Delta_{b^+} \alpha \land \ninv \Delta_{b^+} \alpha$, which implies $\beta$ by (i). From this $(\chi_a \alpha \land \chi_b \alpha) \to \beta$ is a theorem, for any $\beta$. 
		
		\item[(x)] Let $\Gamma = \{\Delta_a \alpha \to \Delta_b \beta, \,\ninv \Delta_b \beta \}$. By (i) it is easy to see that $\Gamma, \Delta_a\alpha \vdash \ninv \Delta_a \alpha$. Clearly $\Gamma, \ninv \Delta_a\alpha \vdash \ninv \Delta_a \alpha$ and so, by proof by cases, $\Gamma, \Delta_a \vee \ninv \Delta_a\alpha \vdash \ninv \Delta_a \alpha$. From this, $\Gamma \vdash \ninv \Delta_a \alpha$, by (Ax6).
		By the deduction-detachment theorem, $(\Delta_a \alpha \to \Delta_b \beta) \to (\ninv \Delta_b \beta \to \ninv \Delta_a \alpha)$ is a theorem. The proof that $(\ninv \Delta_b \beta \to \ninv \Delta_a \alpha) \to (\Delta_a \alpha \to \Delta_b \beta)$ is a theorem is analogous, but now by considering the set $\Gamma'=\{\ninv \Delta_b \beta \to \ninv \Delta_a \alpha, \,\Delta_a\}$.
		
		\item[(xi)] 
		W.l.o.g., we can assume $a \leq b$. Suppose also that $0< a \leq b < 1$. Then, 		
		$\chi_a\alpha \land \chi_b\beta =   \Delta_a \alpha \land \ninv \Delta_{a^+}  \alpha \land \Delta_b \beta \land \ninv \Delta_{b^+}  \beta$. Since $a \leq b$ then $a^+ \leq b^+$. By (Ax7) and  {\sf CPL}, $\Delta_{b^+}\alpha \to \Delta_{a^+}\alpha$ is a theorem. By (x), $\ninv\Delta_{a^+}\alpha \to \ninv\Delta_{b^+}\alpha$ is a theorem. Using this, (Ax8) and  {\sf CPL}, $\chi_a\alpha \land \chi_b\beta$ implies
		$\Delta_b (\alpha \lor \beta) \land \ninv \Delta_{b^+}  \alpha \land \ninv \Delta_{b^+}  \beta$. This implies
		$\Delta_b(\alpha \lor \beta) \land \ninv ( \Delta_{b^+}  \alpha \lor  \Delta_{b^+}  \beta)$, which implies
		$\Delta_b(\alpha \lor \beta) \land \ninv \Delta_{b^+}  (\alpha \lor  \beta) = \chi_b( \alpha \lor  \beta)$, by (Ax8), (Ax1) and  {\sf CPL}. The cases involving $a=0$ or $b=1$ can be proved analogously, and are left to the reader.
		
		\item[(xii)] 
		%
		If $a=0$  the result follows by (Ax9), namely: $\ninv \Delta_{1/n}\alpha$ is equivalent to $\Delta_1 \ninv \alpha$. If $a=1$ then $\chi_1\alpha=\Delta_1\alpha$, which is equivalent to $\Delta_1\ninv\ninv\alpha$, by (iv). By the first part of the proof of this item, this is equivalent to  $\ninv \Delta_{1/n}\ninv\alpha$, that is, $\chi_0 \ninv\alpha$.
		Now, suppose that $0 < a < 1$. Then,
		$\chi_a \alpha  = \Delta_a \alpha \land \ninv \Delta_{a^+} \alpha$. Observe that, since $a=\neg \neg a$, $\Delta_{a} \alpha$ is equivalent to $\ninv \Delta_{(\neg a)^+} \ninv \alpha$, by (Ax9) and item (iv). By (Ax9) again, $\ninv \Delta_{a^+} \alpha$ is equivalent to $\Delta_{\neg a} \ninv \alpha$. Therefore, by  {\sf CPL}, $\chi_a \alpha$ is equivalent to $\ninv \Delta_{(\neg a)^+} \ninv \alpha \,\land\, \Delta_{\neg a} \ninv \alpha$, that is, to $\chi_{\neg a} \ninv \alpha$. 
		
		\item[(xiii)] 
		Note first that $(\Delta_{({*}a)^+} {\star}\alpha  \to \Delta_{a^+} \alpha) \leftrightarrow (\ninv \Delta_{a^+}\alpha \to \ninv \Delta_{({*}a)^+} {\star}\alpha)$, by item (x). Then, from (Ax11), (Ax12) and {\sf CPL}, we get  $(\Delta_a \alpha \to \Delta_{{*}a} {\star}\alpha) \land (\ninv \Delta_{a^+}\alpha  \to \ninv \Delta_{({*}a)^+} {\star}\alpha)$ is a theorem. By using {\sf CPL} once again, we get that the latter formula implies $(\Delta_a \alpha \land \ninv \Delta_{a^+}\alpha) \to (\Delta_{{*}a} {\star}\alpha \land \ninv \Delta_{({*}a)^+} {\star}\alpha)$. From this,  $\chi_a \alpha \rightarrow \chi_{{*}a}{\star}\alpha$ is a theorem, by definition. 
		
		
		\item[(xiv)] Immediate from (xii) and (xiii) and the definition of the $\Delta_a$'s operations and connectives as sequences of $\star$'s and $\ninv$'s. 
		
		\item[(xv)] By the proof of item (vii), $\bigvee_{b \geq 0} \; \chi_b \alpha$ is equivalent to $\Delta_1\alpha \,\vee\, \Delta_{(n-1)/n} \alpha \,\vee\, \ldots \,\vee\, \Delta_{1/n} \alpha \,\vee\, \ninv\Delta_{1/n} \alpha$. Thus, if $a=0$ then the result holds, since $\Delta_{0} \alpha$ and  $\bigvee_{b \geq 0} \; \chi_b \alpha$ are both theorems. Suppose now that $a = k/n > 0$. By reasoning as in item (viii) it is easy to prove that $\bigvee_{b \geq a} \; \chi_b \alpha$ is equivalent to $\Delta_1\alpha \,\vee\, \Delta_{(n-1)/n} \alpha \,\vee\, \ldots \,\vee\, \Delta_{k/n} \alpha$. By (Ax7) and {\sf CPL} it follows that $\Delta_b\alpha \to \Delta_{k/n} \alpha$ is  a theorem, for $b \geq a$. From this $\bigvee_{b \geq a} \; \chi_b \alpha$ is equivalent to $\Delta_{k/n} \alpha$, by {\sf CPL}.
		
		
		\item[(xvi)] It directly follows by an iterative application of (Ax7). 
		
	\end{itemize}
\end{proof}

The following shows that the logic ${\sf AX}^*_{n+1,i}$ proves two basic properties of the unary connective $*$: that $\star\alpha$ is stronger than $\alpha$ and that $\star$ preserves the ordering given by $\Rightarrow_c$. 

\begin{proposition} The following formulas are theorems of ${\sf AX}^*_{n+1,i}$:\\[1mm]
	(1) $\star\alpha \Rightarrow_c \alpha$;\\[1mm]
	(2) $(\alpha \Rightarrow_c \beta) \to (\star\alpha \Rightarrow_c \star\beta)$.
\end{proposition}
\begin{proof}
	(1) From Lemma 4.12(xiii), $\chi_a(\alpha) \vdash_{{\sf AX}^*_{n+1,i}} \chi_{* a}(\star\alpha)$. By {\sf CPL}, it follows that $\chi_a(\alpha) \vdash_{{\sf AX}^*_{n+1,i}} \bigvee_{b \geq * a}\chi_b(\alpha)$. But $\bigvee_{b \geq * a}\chi_b(\alpha) \vdash_{{\sf AX}^*_{n+1,i}} \Delta_{* a}(\alpha)$,  by Lemma 4.12(xv),  hence $\chi_a(\alpha) \vdash_{{\sf AX}^*_{n+1,i}} \Delta_{* a}(\alpha)$. By {\sf CPL}, $\chi_a(\alpha) \vdash_{{\sf AX}^*_{n+1,i}} \chi_{* a}(\star\alpha) \land \Delta_{* a}(\alpha)$. By using {\sf CPL} once again, $\chi_a(\alpha) \vdash_{{\sf AX}^*_{n+1,i}} \bigvee_{b}\chi_{b}(*\alpha) \land \Delta_{b}(\alpha)$, that is, $\chi_a(\alpha) \vdash_{{\sf AX}^*_{n+1,i}} \star\alpha \Rightarrow_c \alpha$.  Using proof-by-cases, $\bigvee_{a} \chi_a(\alpha) \vdash_{{\sf AX}^*_{n+1,i}} \star\alpha \Rightarrow_c \alpha$. But then $\vdash_{{\sf AX}^*_{n+1,i}} \star\alpha \Rightarrow_c \alpha$, by Lemma 4.12(viii).\\[1mm]
	(2) By Lemma 4.12(xiii), (Ax11), and by {\sf CPL}, $\chi_a(\alpha) \land \Delta_a(\beta)\vdash_{{\sf AX}^*_{n+1,i}} \chi_{* a}(\star\alpha) \land \Delta_{* a}(\star \beta)$. By {\sf CPL}, $\chi_a(\alpha) \land \Delta_a(\beta)\vdash_{{\sf AX}^*_{n+1,i}} \bigvee_{b}\chi_{b}(\star\alpha) \land \Delta_{b}(\star \beta)$, that is, $\chi_a(\alpha) \land \Delta_a(\beta)\vdash_{{\sf AX}^*_{n+1,i}} (\star\alpha \Rightarrow_c \star\beta)$. Using proof-by-cases and the definition of $\Rightarrow_c$ it follows that $(\alpha \Rightarrow_c \beta) \vdash_{{\sf AX}^*_{n+1,i}} (\star\alpha \Rightarrow_c \star\beta)$. The result follows by the deduction-detachment theorem w.r.t. $\to$. 
\end{proof}

Next we prove that Boolean formulas behave as in classical propositional logic. First we need a previous lemma with some further provabilities in ${\sf AX}^*_{n+1,i}$. 

\begin{lemma}  \label{boolean}
	(1) ${\sf AX}^*_{n+1,i}$ proves $\star  \alpha \to  \alpha$. 
	
	(2) If $\alpha$ is Boolean, then ${\sf AX}^*_{n+1,i}$ proves $\chi_0 \alpha \lor \chi_1 \alpha$. 
	
	(3) Further, if $\alpha$ is Boolean, then ${\sf AX}^*_{n+1,i}$ proves $\alpha \to \star \alpha$.
	
\end{lemma}

\begin{proof}
	
	(1) By definition $\star  \alpha \to  \alpha = \ninv \Delta_{i/n} {\star} \alpha \lor \alpha$. We reason by cases: 
	
	Let $a\geq i/n$. Then  $\chi_a \alpha \vdash \Delta_{i/n} \alpha$, and $\vdash  \Delta_{i/n} \alpha  \leftrightarrow \alpha$, therefore, 
	$\chi_a \alpha \vdash {\ninv \Delta_{i/n} {\star} \alpha} \lor \alpha$. 
	
	Let $a < i/n$. Then  $\chi_a \alpha \vdash  \chi_{*a} \star \alpha$, and $ \chi_{*a} {\star} \alpha = \Delta_{*a} {\star} \alpha \land \ninv \Delta_{(*a)^+} \star \alpha$. 
	But an easy computation shows that if $a < i/n$, then $(*a)^+ \leq i/n$, and hence, by (Ax7) and (x) of Lemma \ref{theorems}, we have that $\ninv \Delta_{(*a)^+} {\star} \alpha \vdash \ninv \Delta_{i/n} {\star}\alpha$. By {\sf CPL}, we have therefore $\chi_a \alpha \vdash \ninv \Delta_{i/n} {\star}\alpha \lor \alpha$. 
	
	Finally, by (viii) of Lemma \ref{theorems} we get the desired result. \\
	
	\noindent (2)  By induction. If $\alpha = \Delta_a \beta$ (base case),  observe that $\chi_a \beta \vdash \chi_{\Delta(a)} \Delta \beta$, but $\Delta(a) \in \{0, 1\}$, hence  $\chi_a \beta \vdash \chi_{0} \Delta \beta \lor  \chi_{1} \Delta\beta$. The other cases are proved analogously, noticing that all connectives are closed on the set of classical values $\{0, 1\} \subseteq \L_{n+1}$.  \\

	\noindent (3) We have to prove that, if $\alpha$ is Boolean, then ${\sf AX}^*_{n+1,i}$ proves $\varphi = \alpha \to \star \alpha$. 
	
	We prove it by induction. 
	\begin{itemize}
		\item $\alpha = \Delta_a \beta$ (base case). Then we have to prove $\varphi =   \Delta_a \beta \to \star \Delta_a \beta$. We first prove the case $a = 1$, i.e. 
		$\varphi = \Delta_1 \beta \to \star \Delta_1 \Delta_1 \beta$, that is equivalent by (Ax5) to  $\Delta_1 \beta \to \star  \Delta_1 \beta$.  But by definition $\Delta_1 = \star \stackrel{n}{\dots} \star$, and hence $\star \Delta_1 = \Delta_1 \star$, therefore $\varphi$ is equivalent to $\Delta_1 \beta \to \star \beta$. But we have the following chain of derivations in ${\sf AX}^*_{n+1,i}$: $\Delta_1 \beta \vdash \chi_1 \beta \vdash \chi_{*1} {\star} \beta = \chi_{1} {\star} \beta \vdash \Delta_1 {\star} \beta = {\star} \Delta_1 \beta$.  
		
		Now let $a < 1$. Then $\Delta_a \beta$ is equivalent to $\Delta_1 \Delta_a \beta$, i.e. $\star \stackrel{n}{\dots} \star \Delta_a \beta$, and now using repeatedly (1) above $n-1$ times, it follows that  $\star\stackrel{n}{\dots} \star \Delta_a \beta$ implies $\star \Delta_a \beta$. 
		
		\item $\alpha = \ninv \beta$, with $\beta$ Boolean.  Then $\varphi =  \ninv \Delta_{i/n} \ninv \beta \lor \star \ninv \beta$. By (2), ${\sf AX}^*_{n+1,i}$ proves $\chi_0 \beta \lor \chi_1 \beta$. Then $\chi_0 \beta \vdash \chi_1 \ninv \beta$ and then $\chi_0 \beta \vdash \chi_1 \star \ninv \beta$  as well. The latter implies $\Delta_{i/n}  \star \ninv \beta$, that in turn implies $ \star \ninv \beta$  by (Ax10). On the other hand, $\chi_1 \beta \vdash  \chi_0 \ninv \beta$ by (xii) of Lemma \ref{theorems}, and then  by 1),  $\chi_1 \beta \vdash \chi_0 \Delta_{i/n} \ninv \beta$, and hence  $\chi_1 \beta \vdash \chi_1 \ninv \Delta_{i/n} \ninv \beta$, and thus $\chi_1 \beta \vdash \Delta_1\ninv \Delta_{i/n} \ninv \beta$. By (Ax9), the latter is equivalent to $\ninv \Delta_{1/n} \Delta_{i/n} \ninv \beta$, and in turn equivalent to $\ninv \Delta_{i/n} \ninv \beta$ by (Ax5) and (Ax1). 
		
		\item The remaining cases $\alpha = \beta \lor \gamma$, with $\beta, \gamma$ Boolean and $\alpha = \star \beta$ with $\beta$ Boolean can be proved  by cases in a similar way. 
	\end{itemize}
	
\end{proof}

\begin{proposition}
	The sublanguage of Boolean formulas obeys the axioms of classical propositional logic.
\end{proposition}

\begin{proof} Since all the formulas obey the axioms of CPL$^+$, over $(\land, \lor, \to)$, it is enough to check that, if $\alpha$ and $\beta$ are Boolean formulas, then the formula $(\alpha\to \ninv \beta)\to (\beta \to \ninv \alpha)$  is a theorem of ${\sf AX}^*_{n+1,i}$. We first prove by induction that (Ax5) can be generalized to  \\
	
	(Ax5') $\Delta_a \alpha \leftrightarrow \alpha$, if $\alpha$ is Boolean. \\
	
	\noindent The Base case is axiom (Ax5). Then we consider the following inductive steps:
	\begin{itemize}
		\item $\alpha = \ninv \beta$. In this case $\Delta_a \alpha = \Delta_a \ninv \beta$, and, replacing $\neg a$ by $a$ in (Ax9), we get that the latter is equivalent to  $\ninv \Delta_{(\neg a)^+}\beta$, and by I.H., this is equivalent to $\ninv\beta$. 
		
		\item $\alpha = \beta_1 \lor \beta_2$. In this case, $\Delta_a \alpha = \Delta_a (\beta_1 \lor \beta_2)$, that by (Ax8) is equivalent to $(\Delta_a \beta_1) \lor (\Delta_a \beta_2)$, and by I.H., this is equivalent to $\beta_1 \lor \beta_2$. 
		
		\item $\alpha = \star \beta$.  In this case $\Delta_a \alpha = \Delta_a {\star}\beta$. Let $b$ the smallest element of $\L_{n+1}$ such that $a \leq (*b)^+$, then, by (Ax12),  $\Delta_a {\star}\beta$ is equivalent to $\Delta_{b^+} \beta$, and by I.H., this is equivalent to $\beta$, and by (1) and (3) of Lemma \ref{boolean}, $\beta$ is equivalent to $\star \beta$. 
		
	\end{itemize}
	
	Then let $\alpha$ and $\beta$ be Boolean. By definition, $ \alpha \to \ninv  \beta = \ninv \Delta_{i/n} \alpha \lor \ninv \beta$, and due to the above  (Ax5'),  the latter is equivalent to $ \ninv \alpha \lor \ninv \Delta_{i/n} \beta$, that, by definition is in fact, $\beta \to \ninv \alpha$. 
\end{proof}

Finally we prove soundness and completeness of the logic ${\sf AX}^*_{n+1,i}$. 

\begin{proposition} [Soundness of ${\sf AX}^*_{n+1,i}$] The calculus ${\sf AX}^*_{n+1,i}$ is sound w.r.t. the matrix $\langle \bL_{n+1}^*, F_{i/n}\rangle$, that is: $\Gamma \vdash_{{\sf AX}^*_{n+1,i}} \varphi$ implies that $\Gamma \vDash_{\langle \bL_{n+1}^*, F_{i/n}\rangle} \varphi$, for every finite set of formulas $\Gamma \cup \{\varphi\}$.
\end{proposition}
\begin{proof}
	Straightforward, taking into account the definitions of the terms $\Delta_a$'s and $\chi_a$'s in Proposition \ref{prop3}. 
\end{proof}

Since ${\sf AX}^*_{n+1,i}$ is a finitary Tarskian logic, completeness can be proved by using maximal consistent sets of formulas. Thus, as a consequence of the well-known Lindenbaum-{\L}os theorem, if $\Gamma \nvdash_{{\sf AX}^*_{n+1,i}} \varphi$ then  $\Gamma$ can be extended to a maximal set  $\Lambda$ such that $\Lambda  \nvdash_{{\sf AX}^*_{n+1,i}} \varphi$. We will call the set $\Lambda$ {\em maximal \ntwrt\varphi}  in  ${\sf AX}^*_{n+1,i}$. 

In the following proposition, we list the main properties of maximal consistent sets in  ${\sf AX}^*_{n+1,i}$.


\begin{proposition} \label{prop-maxim2} 
	Let $\Lambda$ be a set of formulas which is maximal non-trivial w.r.t. some formula $\varphi$ 
	in  ${\sf AX}^*_{n+1,i}$. Then:%
	\begin{enumerate}
		\item $\Lambda$ is closed, i.e., $\Lambda \vdash_{{\sf AX}^*_{n+1,i}}\psi$ iff $\psi\in\Lambda$,  for every formula $\psi$
		\item $\alpha \vee \beta \in \Lambda$ \ iff \ either $\alpha \in \Lambda$ or $\beta \in \Lambda$
		\item $\alpha \wedge \beta \in \Lambda$ iff $\alpha,\beta \in \Lambda$
		\item ${\finv}_{i/n} \alpha \in \Lambda$ iff $\alpha \notin \Lambda$
		\item $\alpha \rightarrow \beta \in \Lambda$ iff either $\alpha \notin \Lambda$ or $\beta \in \Lambda$
		\item  $\alpha \leftrightarrow \beta \in \Lambda$ iff either $\alpha,\beta \in \Lambda$ or $\alpha,\beta \notin \Lambda$
		\item For every formula $\gamma$, one and only one of the conditions  `$\chi_{a} \gamma   \in \Lambda$', holds, for $a \in \L_{n+1}$.
		\item  $\chi_a \alpha \in \Lambda$ iff $\chi_{\neg a}\ninv\alpha \in \Lambda$.
		\item  If $\chi_a \alpha \in \Lambda$ then $\chi_{*a} \star \alpha \in \Lambda$.
	\end{enumerate}
\end{proposition}
\begin{proof}
	\begin{enumerate}
		\item This holds by construction of the maximal consistent sets
		\item  The `only if' part  follows by the fact that $\Lambda$ is maximal non-trivial w.r.t. $\varphi$, and by taking into account   that ${\sf AX}^*_{n+1,i}$ satisfies proof by cases (recall the observations after Definition~\ref{calLin}). Indeed, if $\alpha \notin \Lambda$ and $\beta \notin \Lambda$ then $\Lambda,\alpha \vdash \varphi$ and $\Lambda,\beta \vdash \varphi$, hence $\Lambda,\alpha \vee \beta \vdash \varphi$. From this, $\alpha \vee \beta \notin \Lambda$.
		
		\item In order to prove that $\alpha \land \beta=\ninv(\ninv\alpha \vee \ninv \beta) \in \Lambda$ implies that $\beta \in \Lambda$ it is necessary to use (Ax1), showing that $\ninv(\ninv\beta \vee \ninv \alpha) \in \Lambda$, and so apply (Ax3) and (Ax2).
		
		\item Suppose ${\finv}_{i/n}\alpha \in \Lambda$, i.e. $\ninv \Delta_{i/n}\alpha \in \Lambda$. Then,  by (i) of Lemma \ref{theorems} it follows that $\Delta_{i/n} \alpha \notin \Lambda$ and,  by (v) of the same Lemma, it must be $\alpha \notin \Lambda$ as well. Conversely, assume ${\finv}_{i/n}\alpha \notin \Lambda$, that is, $\ninv \Delta_{i/n} \alpha \not\in \Lambda$.  Then, by (Ax6), $\Delta_{i/n} \alpha \in \Lambda$, and hence $\alpha\in \Lambda$, by (Ax11). That is: $\alpha\notin \Lambda$ implies that ${\finv}_{i/n} \alpha \in \Lambda$.
		
		\item By definition,  $\alpha \rightarrow \beta = {\finv}_{i/n} \alpha \lor \beta$.  Then, by item (2), $\alpha \to \beta \in \Lambda$ iff  either ${\finv}_{i/n} \alpha \in \Lambda$ or $\beta \in \Lambda$, iff  either $\alpha \not\in \Lambda$ or $\beta \in \Lambda$, by (4).
		
		\item Easily follows from (3) and (5).
		
		\item  By (viii) of Lemma~\ref{theorems} and by (1), $\bigvee_{0 \leq a \leq 1} \chi_{a} \gamma \in \Lambda$. By (2), $\chi_{a} \gamma \in \Lambda$ for some $a \in \L_{n+1}$. By (ix) of Lemma~\ref{theorems}, there are no $a \neq b$ such that $\chi_{a} \gamma, \chi_{b}\gamma \in \Lambda$, since $\varphi \notin \Lambda$. From this, $\chi_{a} \gamma \in \Lambda$ for one and only one $a \in \L_{n+1}$. 
		
		\item  It follows from (xii) of Lemma~\ref{theorems} and by (5).
		
		\item If directly follows from (xiii) of Lemma \ref{theorems} together with (5).
	\end{enumerate}
\end{proof}

\begin{lemma} [Truth Lemma for ${\sf AX}^*_{n+1,i}$] \label{TL} Let $\Lambda$ be a maximal set of formulas \ntwrt\varphi\  in  ${\sf AX}^*_{n+1,i}$. Consider the following mapping $e_\Lambda$ of formulas to $\L_{n+1}$ defined as follows: for each formula $\gamma$, 
	$$
	e_\Lambda(\gamma) = a \ \mbox{ if } \  \chi_{a} \gamma   \in \Lambda. $$
	Then,  $e_\Lambda$ is a $\langle \bL_{n+1}^*, F_{i/n}\rangle$-evaluation. 
\end{lemma}
\begin{proof} First, observe that $e_\Lambda$ is well-defined, i.e. every formula gets a unique value. This is an immediate consequence of (7) of Proposition~\ref{prop-maxim2}. We have to prove that the following conditions are satisfied for every formulas $\alpha$ and $\beta$:
	
	\begin{itemize}
		
		\item[(i)] $e_\Lambda(\alpha \lor \beta) = \max(e_\Lambda(\alpha), e_\Lambda(\beta))$.
		Indeed, let $c = e_\Lambda(\alpha \lor \beta)$. By definition, $\chi_c(\alpha \lor \beta) \in \Lambda$, and so  $\chi_c(\alpha) \lor \chi_c(\beta) \in \Lambda$,  by (vii) of Lemma \ref{theorems}. By (2) of Proposition~\ref{prop-maxim2}, either $\chi_c(\alpha) \in \Lambda$ or  $\chi_c(\beta) \in \Lambda$. That is, either $e_\Lambda(\alpha) = c$ or $e_\Lambda(\beta) = c$. By way of contradiction, suppose e.g. $e_\Lambda(\alpha) = d > c$ and  $e_\Lambda(\beta) = c$. Then $\chi_c(\alpha) \in \Lambda$ and $\chi_d(\alpha) \in \Lambda$ and  so, by  (xi) of Lemma \ref{theorems}, $\chi_d(\alpha \lor \beta) \in \Lambda$. Hence $e_\Lambda(\alpha \lor \beta) = d  > c$, contradiction. From this, $d \leq c$ and $c = \max(e_\Lambda(\alpha), e_\Lambda(\beta))$. 
		
		
		\item[(ii)] $e_\Lambda(\ninv \alpha) = 1- e_\Lambda(\alpha)$.
		Indeed, let $c = e_\Lambda(\alpha)$, that is, $\chi_c \alpha  \in \Lambda$. By (8) of Proposition~\ref{prop-maxim2}, $\chi_{1-c} \ninv\alpha \in \Lambda$, i.e. $e_\Lambda(\ninv \alpha) = 1-c$. 
		
		\item[(iii)]  $e_\Lambda(\star \alpha) = *(e_\Lambda(\alpha))$. 
		Indeed, let $c = e_\Lambda(\alpha)$. By definition, $\chi_c \alpha  \in \Lambda$. By (9) of Proposition~\ref{prop-maxim2}, $\chi_{{*}c} {\star}\alpha \in \Lambda$, i.e. $e_\Lambda({\star}\alpha) = {*}c$.
		
	\end{itemize}
	This ends the proof. 
\end{proof}

Finally, we can state the completeness result for ${\sf AX}^*_{n+1,i}$. 

\begin{theorem} [Completeness of ${\sf AX}^*_{n+1,i}$] \label{complLin} The calculus ${\sf AX}^*_{n+1,i}$ is complete w.r.t. $\langle \bL_{n+1}^*, F_{i/n}\rangle$, that is: $\Gamma \vDash_{\langle \bL_{n+1}^*, F_{i/n}\rangle} \varphi$ implies that $\Gamma \vdash_{{\sf AX}^*_{n+1,i}} \varphi$, for every finite set of formulas $\Gamma \cup \{\varphi\}$.
\end{theorem}
\begin{proof}
	Let $\Gamma\cup \{\varphi\}$ be a set of formulas  of ${\sf AX}^*_{n+1,i}$	such that $\Gamma\nvdash_{{\sf AX}^*_{n+1,i}} \varphi$.
	By Lindenbaum-\L os, there exists a set $\Lambda$  maximal \ntwrt\varphi\ in ${\sf AX}^*_{n+1,i}$
	such that $\Gamma \subseteq \Lambda$.
	Let $e_\Lambda$ be the evaluation defined as in the Truth Lemma~\ref{TL}.
	Then, it follows that, for every formula $\gamma$: $e_\Lambda(\gamma) \in F_{i/n}$ iff $\chi_1\gamma \,\vee\, \chi_{(n-1)/n}\gamma \,\vee\, \ldots \,\vee\, \chi_{i/n}\gamma \in \Lambda$, by the Truth Lemma~\ref{TL}. Moreover, by (xiv) of Lemma \ref{theorems}, this is equivalent to the condition 
	$\Delta_{i/n}\gamma \in \Lambda$. By (Ax10) and by (v) of Lemma~\ref{theorems}, the latter is equivalent to the condition $\gamma \in \Lambda$. That is, for every formula $\gamma$ we have: $e_\Lambda(\gamma) \in F_{i/n}$ iff $\gamma \in \Lambda$. Therefore, $e_\Lambda$ is an evaluation such that $e_\Lambda[\Gamma] \subseteq F_{i/n}$ but $e_\Lambda(\varphi) \notin F_{i/n}$, since $\varphi \not\in \Lambda$. This means that $\Gamma \not\vDash_{\langle \bL_{n+1}^*, F_{i/n}\rangle} \varphi$. 
\end{proof}



\section{Subalgebras of $\bL_{n+1}^*$ and G\"odel algebras with an involutive negation and a $\star$ operation} \label{sec:representable}

In this  section we present an alternative approach to capture the behaviour of the square operator in  structures obtained by adding a unary operator $\star$ to G\"odel chains with an involutive negation.  In order to do so, in a first subsection we characterize the subalgebras of a $\bL_{n+1}^*$ algebra. 
The second subsection is devoted to the study of structures obtained by adding a unary operation $\star$ to G\"odel algebras with an involutive negation in general.  
There, we will axiomatically characterize the class of those algebras, that we will call {\em representable}, 
whose implication free-reducts are isomorphic to subalgebras of the $\L_{n+1}^*$'s.

From now on, we will denote by $[0,1]_{MV}^*$ the algebra defined over the real unit interval by the \luk\ operations $\vee, \neg, \ast$. 

\subsection{Finite subalgebras of $[0,1]_{MV}^*$}\label{subsec:subLestrella1}

We start by noticing that, for every $n >1$, $\bL_{n+1}^*$ and its subalgebras are subalgebras of $[0,1]_{MV}^*$. Conversely, as it will be shown in Proposition \ref{prop:generated}, every finite subalgebra of $[0,1]_{MV}^*$ is a subalgebra of some $\bL_{n+1}^*$  (although, as seen in Example \ref{exemp1}(1), it is not necessarily of the form $\bL^*_{m+1}$ for some $m \leq n$).  Then  studying  the subalgebras of  $\bL_{n+1}^*$ (for any $n > 1$) turns out to be  equivalent to study the finite subalgebras of $[0,1]_{MV}^*$. 

For what follows it is useful to introduce the notion of {\em skeleton} of an element of a finite subalgebra of $[0,1]^*_{MV}$. In order for the next definition to be precise, let us notice that  Definition \ref{remProcP}, introducing the procedure ${\tt P}$, can be easily adapted to any finite subalgebra ${\bf A}$ of $[0,1]^*_{MV}$.
\begin{definition}
	Let ${\bf A}$ be a finite subalgebra of $[0,1]^*_{MV}$, let $a$ be a positive element of $A\setminus\{1\}$ and let ${\tt P}({\bf A}, a)=[a_1,\ldots, a_k]$, with $a_{k+1} = a_j$ for some $1 \leq j \leq k$. Then we define the {\em skeleton of $a$ in ${\bf A}$}, denoted by $Sk({\bf A}, a)$ as the finite string of symbols $[o_1,\ldots, o_{k}]$, where $o_i\in \{\ast, \msim\}$ is such that  $o_i(a_i)=a_{i+1}$ for all $i=1,\ldots,k$ and thus $o_{k}(a_k) = a_j$.  
\end{definition}

Due to the way procedure $\tt P$ is defined, one can notice that any skeleton will be a string of symbols of the form
$$ [*^{n_1}, \ninv, *^{n_2}, \ldots, \ninv, *^{n_k} ],$$
with $k > 1$ and $n_1, \ldots, n_{k-1} > 0$,  where $*^{n_i}$ is a shorthand for `$*, \stackrel{n_i}{\ldots}, *$', i.e.\ the string with $n_i$ repetitions of $*$. Moreover, if $n_k = 0$ then we assume the sk-sequence reduces to $ [*^{n_1}, \ninv, *^{n_2}, \ldots, \ninv ]$. In what follows, we will call this kind of strings {\em sk-sequences}.

Let us notice that, as in the case of $\bL^*_{n+1}$-algebras, if ${\bf A}$ is a finite strictly simple subalgebra of $[0,1]^*_{MV}$ and $\co$ is the coatom of ${\bf A}$, then ${\tt P}({\bf A}, \co)=[a_1,\ldots, a_k]$ is such that $a_k=\msim \co$, i.e.,  ${\tt P}({\bf A}, \co)$ ends with the atom of ${\bf A}$. Thus, $Sk({\bf A}, \co)=[o_1,\ldots, o_k]$ is such that $o_k=\msim$. 

The following result presents a slight generalization of the above argument.

\begin{proposition}\label{strictly}
	A finite subalgebra ${\bf A}$ of $[0,1]_{MV}^*$ is strictly simple iff ${\bf A} = \langle a \rangle^*$ for a positive element $a \in A$ and ${\tt P}({\bf A},a) = [a_1, \ldots,a_k]$ with $a_{k+1} = a$. 
\end{proposition}
\begin{proof}
	Left-to-right. It is obvious that if ${\bf A}$ is strictly simple, then for any $a \in A \setminus \{0, 1\}$, $A = \langle a \rangle^*$ (otherwise $\langle a \rangle^*$ would be a proper subalgebra of $\bf A$). Moreover, if $A = \langle a \rangle^*$ but ${\tt P}({\bf A},a) = [a_1, \ldots, a_k]$ with $a_{k+1} = a_i$ for $i > 1$, then $\langle a_i \rangle^* \subsetneq A$ (since $a \not\in \langle a_i \rangle^*$) and $\bf A$ would not be strictly simple.
	
	Right-to-left. If ${\bf A} = \langle a \rangle^*$ for some positive element $a \in A\setminus\{1\}$ and ${\tt P}({\bf A},a) = [a_1, \ldots ,a_k]$ with $a_{k+1} = a$, then every positive element of $\bf A$ belongs to ${\tt P}({\bf A},a)$ and, since $a_{k+1} = a$ for any $a_i$ we have ${\tt P}({\bf A},a_i) = [b_1, \ldots,b_k]$ with $a_i = b_1 = b_{k+1}$, i.e.\ ${\tt P}({\bf A},a_i)$ is a cyclic permutation of the sequence ${\tt P}({\bf A},a)$. Therefore, for any positive element $a_i \in A$, ${\bf A} = \langle a_i \rangle^*$ and $\bf A$ has no subalgebras, i.e.,  it is strictly simple. 
\end{proof}


\begin{example}\label{ex:new2}
	Consider Example \ref{exemp1} (1). There we have $$ {\bf A} = \langle 8/9 \rangle^*=\{0,1/9, 2/9, 4/9, 5/9, 7/9, 8/9,1\}$$
	with $\co = 8/9$ and ${\tt P}({\bf A},8/9 ) = [8/9,7/9,5/9,1/9]$. Then,  $Sk({\bf A}, 8/9) = [\ast, \ast, \ast, \sim]$. Observe that $8/9$ is the solution of the equation  $\ninv (\ast^3(x)) = x$. Indeed using the semantics of $\ast, \sim$ in $[0,1]_{MV}^*$, the equation $\ninv (\ast^3(x))=x$ can be written as $1 - (2(2(2x-1)-1)-1) = x$ which has a unique solution $x = 8/9$. Notice also that $Sk({\bf A},7/9) = [\ast, \ast, \sim, \ast]$ and $Sk({\bf A},5/9) = [\ast, \sim, \ast, \ast]$  are cyclic permutations of $Sk({\bf A}, \co)$.
\end{example}
The example above anticipates a general result that we are going to prove in the next proposition. Henceforth, if $R = [o_1,\ldots, o_k]$ is  any sequence where every $o_i \in \{*, \ninv\}$, 
we will adopt the notation $f_{R}$ to indicate the unary function in $[0, 1]$ defined as
$$
f_R(x) = o_k(o_{k-1}(\ldots o_1(x)\ldots)) .
$$
In particular,
any finite subalgebra  ${\bf A}$ of $[0,1]^*_{MV}$  and any $a\in A$ will have an associated function $f_{Sk({\bf A},a)}$. 
For instance, taking into account Example \ref{ex:new2} above, one has 
$$
f_{Sk(\langle 8/9\rangle^*, 8/9)}(x)=\msim(\ast(\ast(\ast(x)))),
$$
while
$$
f_{Sk(\langle 8/9\rangle^*, 5/9)}(x)=\ast(\ast(\msim(\ast(x)))).
$$


\begin{proposition}\label{propSolutions} Let $S$ be a sequence of symbols from $\{*, \ninv\}$, and let $f_S$ be its corresponding function defined as above. 
	%
	%
	Then we have:
	\begin{itemize}
		\item[(i)] if $S$ is a sk-sequence, the equation $f_{S}(x) = x$ has a unique and rational solution $x_S > 1/2$; 
		\item[(ii)] the equation $f_S(x) = d$ has a unique and rational solution for every rational number $0 < d < 1$.
	\end{itemize}
\end{proposition}

\begin{proof} First of all, observe that for any sequence $S$, as a function $f_S:[0, 1] \to [0, 1]$, $f_S$ is continuous, and it is increasing if the number of negations $\msim$ involved is even, otherwise it is decreasing. Let us assume then that $f_S$ involves an even number of negations, and hence $f_S$ is increasing with $f_S(0 ) = 0$ and $f_S(1) = 1$. By composing the functions $*$ and $\msim$ in the required form, one can easily check that $f_S$ is of the following form: there are rationals $a, b \in [0, 1]$, with $0 \leq a < b \leq 1$ such that:
	$$f_S(x) = \left \{ 
	\begin{array}{ll}
	0, & \mbox{if } 0 \leq x \leq a\\
	(x-a)/(b-a), &  \mbox{if } a \leq x \leq b \\
	1, & \mbox{if } b \leq x \leq 1\\
	\end{array}
	\right .
	$$
	As for (i), if $S$ is a sk-sequence, by construction, the rational $a$ is such that $1/2 \leq a$. Therefore, it is clear that the equation  $f_S(x) = x$ has as a unique rational solution $x_S = a / (1-b+a)$, satisfying $a < x_S < b$. 

	As for (ii), since $f_S$ is always strictly increasing in the open interval $(a, b)$, the graph $y = f_S(x)$ always intersects the horizontal line $y = d$ if $0 < d < 1$, and hence the equation $f_S(x) = d$ has always as unique solution $x_d = (b-a)d + a$. 
	
	If $f_S$ involves an odd number of negations, then $f_S$ is decreasing, with $f_S(0) = 1$ and $f_S(0) = 1$, and the arguments for (i) and (ii) are completely dual to the ones above. 
\end{proof}


To graphically exemplify the above result, Figure \ref{figEstrellas} displays examples of functions $f_S$ for a sk-sequence $S$ containing odd and even occurrences of $\msim$ and how they intersect the diagonal in a single point. 

\begin{center}
\begin{figure}[h!]
\vspace*{-0.25cm}
  \centering
    \includegraphics[width=0.85\textwidth]{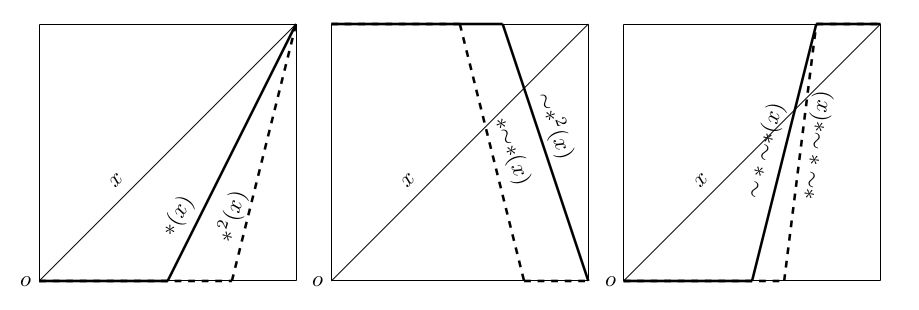} \caption{\small Examples of functions $f_S(x)$ with odd (central figure) and even (right-most figure) occurrences of $\msim$.} \label{figEstrellas}
\end{figure}
\end{center}

\vspace*{-0.75cm}

The following result is an easy consequence of the previous Proposition \ref{propSolutions}.

\begin{corollary}\label{determined}
	Let ${\bf A}$ be a finite strictly simple subalgebra of $[0,1]_{MV}^*$, then
	\begin{enumerate}
		\item[(1)] If $a$ is a positive element of $A \setminus\{1\}$, then $a$ is the unique rational  solution of the equation $f_{Sk({\bf A},a)} (x) = x$.
		\item[(2)] ${\bf A}$ is completely determined by $Sk({\bf A}, \co)$, meaning that for any two different strictly simple subalgebras ${\bf A}$ and ${\bf A}'$ of $[0,1]^*_{MV}$, $Sk({\bf A}, \co)\neq Sk({\bf A}', \co')$, where $\co$ and $\co'$ denote the coatoms of ${\bf A}$ and ${\bf A}'$ respectively.
	\end{enumerate}
\end{corollary}
\begin{proof}
	(1) By Proposition \ref{strictly}, ${\bf A}= \langle a \rangle^*$ and $f_{Sk({\bf A}, a)}(a)=a$. In  other words $a$ is a solution of $f_{Sk({\bf A}, a)}(x)=x$. Thus, by Proposition \ref{propSolutions} (i), $a$ is the unique and rational solution of the equation above. 
	
	(2) Suppose $Sk({\bf A}, \co) = Sk({\bf A}', \co') = S$, this means that $f_S(\co) = \co$ as well as $f_S(\co') = \co'$. But since by (1) the solution of the equation $f_S(x) = x$ is unique, we have that $\co = \co'$.  Since ${\bf A}$ and ${\bf A}'$ are assumed to be strictly simple, we finally have ${\bf A} = \langle \co \rangle^* = \langle \co' \rangle^* = {\bf A}'$.
	%
\end{proof}


In the corollary above, the hypothesis of ${\bf A}$ being strictly simple cannot be relaxed as the following example shows.
%
%
The following example proves that not any sk-sequence 
can be the skeleton of a strictly simple subalgebra of $[0,1]^*_{MV}$.
\begin{example}
	Consider the sk-sequence $S = [\ast, \ast, \sim, \ast, \ast, \sim]$ and suppose it is a skeleton of the coatom $\co$ of strictly simple subalgebra ${\bf A}$ of $[0,1]^*_{MV}$. Then $\co$  is the rational solution of the equation $f_S(x) = x$. However, the equation  $f_S(x)= 1-2(2(1-(2(2x-1)-1)-1)-1) = x$ has as   unique solution $4/5$, and ${\tt P}(\bL_{5+1}^*,4/5) = [4/5,3/5,1/5]$ and hence $Sk(\langle 4/5\rangle^*) = [ \ast, \ast, \sim]$, that is different from the initial sequence $S$. 
\end{example}

Proposition \ref{propSolutions} and Corollary \ref{determined} allows us to prove, as announced above, that the set of all finite subalgebras of $[0, 1]_{MV}^*$ coincides in fact with the set of subalgebras of all the $\bL_{n+1}^*$ algebras. 

\begin{proposition}\label{prop:generated}
	The following conditions hold: 
	\begin{itemize}
		\item[(1)] The subalgebra of $[0,1]_{MV}^*$ generated by an element $a \in [0,1]$ is finite iff $a$ is a rational number.
		\item[(2)] The finite subalgebras of $[0,1]_{MV}^*$ contain only rational numbers.
		\item[(3)]  Any finite subalgebra of $[0, 1]_{MV}^*$ is a subalgebra of some $\bL_{n+1}^*$.
	\end{itemize}
\end{proposition}
\begin{proof}
	(1) Left-to-right. 
	Let $a \in [0,1]$ and assume, without loss of generality, that $a$ is positive, that is, $a > 1/2$ (clearly, if $a$ was not positive one could consider its negation $\neg a>1/2$). 
	
	If $\langle a\rangle^*$ is finite then ${\tt P}({\bf A},a) = [a_1, \ldots,a_k]$, where $a_{k+1} = a_i$ for some $i \leq k$. Then $\langle a_i \rangle^*$ is finite and strictly simple, and by Corollary \ref{determined}, $a_i$ is rational. But since $a_i \in \langle a \rangle^*$, 
	there exists a term $f(x)$ as those considered in Proposition \ref{propSolutions} such that $f(a)= a_i$, and by (ii) of Proposition \ref{propSolutions}, $a$ has to be rational as well.
	
	
	%
	
	Right-to-left. Assume $a = n/d$ is a positive rational number of $[0,1]$. Then, as we already observed in the proof of Lemma \ref{lemmaKey}, 
	the application of either $\ninv$ or $*$ to $a$ produces another rational number in $[0,1]$ that, moreover, has  the same denominator $d$. Indeed, $\ninv a=1-(n/d)=(d-n)/d$ and $*(n/d)=2n/d-1=(2n-d)/d$. Thus $\langle a\rangle^*$ is necessarily finite because there are only $d+1$ rational numbers in $[0,1]$ sharing the same denominator $d$. 
	
	(2) It is an easy consequence of (1).  
	
	(3) Let $\bf A$ be a finite subalgebra of $[0, 1]_{MV}^*$. Then all its elements are rational, and hence there must exist $n$ such that $A \subseteq \{0, 1/n, \ldots, (n-1)/n, 1\}$ (for instance take $n$ as the l.c.m.\ of all denominators appearing in $A$), and therefore $\bf A$ must be a subalgebra of $\bL_{n+1}^*$. 
\end{proof}

In what follows, for every natural number $n$ and every sequence $R$, we will denote by $(n)R$ the concatenation of $R$ with itself $n$-times.

We  say that a sequence $S$ is {\em periodic} if it contains a strict subsequence $R$ such that $S=(n)R$ for some $n \geq 2$. A sequence $S$ will be called {\em non-periodic} if it is not periodic.

\begin{proposition}\label{non-periodical}
	For every finite strictly simple subalgebra ${\bf A}$ of $[0,1]_{MV}^*$, $Sk({\bf A}, \co)$ is non-periodic.  
\end{proposition}
\begin{proof}
	Assume by way of contradiction that $Sk({\bf A}, \co)$ is periodic and hence that there exists a subsequence $R$ of $Sk({\bf A}, \co)$ such that $Sk({\bf A}, \co)=(n)R$ for some $n\geq 2$. Since ${\bf A}$ is strictly simple, by Corollary \ref{determined}, $\co$ is the unique rational solution of $f_{Sk({\bf A},\co)}(x) = x$. Denote it by $r/n$. Now, consider the equation $f_R(x) = x$ and let $k/m$ its unique rational solution. Notice that $ k/m$ is also a solution of the equation $f_{Sk({\bf A},\co)}(x) = x$. In fact, $f_{Sk({\bf A},\co)}(x) = f_R(f_R(\ldots f_R(x)\ldots))$, and since $f_R(k/m) = k/m$, $f_{Sk({\bf A},\co)}(k/m) = k/m$. This implies that $r/n$ and $k/m$ are solutions of the same equation $f_{Sk({\bf A},\co)}(x) = x$. But the solution is unique and so $r/n = k/m$. Therefore, ${\bf A} = \langle k/m \rangle^*$ and $Sk({\bf A}, \co) = R$ while we assumed that $Sk({\bf A}, \co) = (n)R$ for $n\geq 2$. Contradiction.
	%
	%
\end{proof}


Finally, the next proposition presents additional properties of finite subalgebras of $[0,1]^*_{MV}$ that will be useful in the next section.

\begin{proposition}\label{partition}
	Let ${\bf A}$ be a finite subalgebra of $[0,1]_{MV}^*$. Then:
	\begin{enumerate}
		\item For every positive $a \in A\setminus\{1\}$, either ${\bf A}$ is strictly simple (i.e.,  ${\bf A}= \langle a \rangle^*$ and $f_{Sk({\bf A},a)} = a$), or $ \langle a \rangle^*$ contains a unique strictly simple subalgebra.
		\item If ${\bf B},{\bf C}$ are two different strictly simple subalgebras of ${\bf A}$, then $Sk({\bf B},\co_B) \neq Sk({\bf C}, \co_C)$, where $\co_B$ and $\co_C$ respectively denote the coatom of ${\bf B}$ and the coatom of ${\bf C}$.
		\item If ${\bf B}_1,{\bf B}_2,\ldots,{\bf B}_r$ are the strictly simple subalgebras of ${\bf A}$ then $\{B_1^+, ..., B_r^+\}$ is a partition of $A$ where $B_i^+ = \{a \in A : \langle a \rangle \supseteq B_i\}$. 
	\end{enumerate}
\end{proposition}
\begin{proof}
	(1) If ${\bf A}$ is strictly simple, the claim follows from Corollary \ref{determined} (1). Thus assume   ${\bf A}$ is not strictly simple and take a positive $a \in A\setminus\{1\}$.  By Proposition \ref{strictly} this implies that ${\tt P}({\bf A},a) = [a_1=a,...,a_k]$ with $a_{k+1} = a_i$ for some $i > 1$. Then it is obvious that ${\tt P}({\bf A},a_i)=[a_i, ..., a_k]$ with $a_{k+1} = a_i$. Then $ B_a = \langle a_i \rangle^\star$  is strictly simple and $B_a \subsetneq \langle a \rangle^\star$. Moreover by construction, for each $a \in A$ the subalgebra $B_a$ is the unique strictly simple subalgebra contained in  $\langle a \rangle^\star$
	
	(2) is an immediate consequence of Corollary \ref{determined} (2).

	(3) Observe that  (1) and (2) imply that the union of $B_i^+$'s is the whole domain $A$ of the algebra ${\bf A}$. Obviously, two different strictly simple subalgebras $B_i$ and $B_j$ must be disjoint since, if $a \in (B_i \cap B_j)$  by Proposition \ref{strictly} $B_i = B_j =\langle a\rangle^\star$. On the other hand, if $a \in (B_i^+\cap B_j^+)$ by the previous (2) $B_i = B_j$  and thus $B_i^+ = B_j^+$. 
\end{proof}

We end this first subsection with the following observations. 
\begin{remark}
	\mbox{}
	\begin{itemize}
		\item[(1)] Notice that in all finite subalgebras of $[0,1]_{MV}^*$, G\"odel implication is definable as we did for every $\bL_{n+1}^*$-algebra (see Section  \ref{sec:Lstaralgebras}).
		\item[(2)] The logic whose algebraic semantics is  the variety generated by a subalgebra ${\bf A}$ of $[0,1]_{MV}^*$ in the language of $\vee, \neg, \ast$, 
		can be axiomatized following the same method used for $\bL_{n+1}^*$  in Section \ref{sec:logics}. 
	\end{itemize}
\end{remark}
The first remark clearly relates subalgebras of $[0,1]_{MV}^*$ with G\"odel chains with an involutive negation plus an $\ast$ operation. This relation is deepened in the next subsection.

\subsection{Adding a $\star$-operator to involutive G\"odel algebras}\label{subsec:IGStar}

In this subsection we present an alternative approach to capture the behaviour of \luk's square by adding a unary operator $\star$  to a  G\"odel algebra with an involution, from a different algebraic perspective. 

Let us hence define the following structures.

\begin{definition}\label{def:Ginvstar}
	A {\em G\"odel-algebra with an involution $\ninv$ and a $\star$-operator} ({\em$IG^\star$-algebra} for short) is a triple $({\bf A}, \ninv, \star)$ where $({\bf A},\ninv)$ is a G\"odel algebra with involution and $\star$ is a unary operator on $A$ satisfying the following equations:
	\begin{itemize}
		\item[($\star1$)] $(x\vee\msim x)= \Delta(x\Leftrightarrow_G\star x)$;
		\item[($\star2$)] $\Delta(x \Rightarrow_G \msim x) = \neg_G \star x$;
		\item[$(\star3)$] $\msim\Delta(x\Rightarrow_G \msim x) \wedge \msim\Delta x \leq \msim\Delta(x \Rightarrow_G \star x )$;
		\item[$(\star4)$] $\Delta(\msim x \Rightarrow_G  x)\wedge \Delta( \msim y\Rightarrow_G y)\wedge \msim \Delta(x\Rightarrow_G y) \leq\msim \Delta(\star x\Rightarrow_G \star y)$;
		\item[$(\star5)$] $\Delta(\msim x\Rightarrow_G x)\wedge \Delta(\msim y\Rightarrow_G y)\wedge \Delta(\star x\Leftrightarrow_G \star y)\leq \Delta( x\Leftrightarrow_G y)$;
	\end{itemize}
	where $\Delta x$ is an abbreviation for $\neg_G\msim x$. If ${\bf A}$ is a G\"odel algebra in the variety $\mathbb{IG}_{n+1}$ (recall Section \ref{sec :preliminaries}), we will say that $({\bf A},\ninv, \star)$ is an $IG^\star_{n+1}$-algebra.  The varieties of $IG^\star$-algebras and $IG^\star_{n+1}$-algebras will be denoted by $\mathbb{IG}^\star$  and $\mathbb{IG}_{n+1}^\star$  respectively. 
\end{definition}

Let us explain the equations above on an standard $IG$-algebra $([0,1]_G, \msim)$ where $\msim:[0,1]\to[0,1]$ is an involution with fixpoint $1/2$. Let us first of all recall that in $([0,1]_G,\msim)$ the following conditions hold for all $x$: $x\vee\msim x=1$ iff either $x=1$ or $x=0$; $x\Rightarrow_G \msim x = 1$ iff $x\leq\msim x$ and hence $x$ is negative, that is,  $x\leq 1/2$; and $\msim \Delta(x\Rightarrow_G \msim x)=1$ iff $\msim x< x$, and hence $x$ is a strictly positive, meaning that  $x> 1/2$. 

\begin{itemize}
	\item[($\star1$)] Since $\Delta z\in \{0,1\}$ for all $z\in [0,1]$, and for no $x\in [0,1]$, $x\vee\msim x=0$, the  formula $(x\vee\neg x)= \Delta(x\Leftrightarrow_G\star x)$ states that $x=\star x$ iff either $x=0$ or $x=1$.
	
	\item[($\star2$)] As we recalled above, $x \Rightarrow_G  \msim x=1$ iff $ x\leq \msim x$ iff $x\leq 1/2$. Thus, $\Delta(x \Rightarrow_G \msim x)=1$ if $x\leq 1/2$ and it is $0$ otherwise. Moreover $\neg \star x=1$ if $\star x=0$ and $\neg \star x=0$ if $\star x>0$. Therefore $\Delta(x \Rightarrow_G \msim x) = \neg\star x$ states that $x \leq 1/2$ iff $\star x=0$. Further notice that ($\star2$) is equivalent to $\msim\Delta(x\Rightarrow_G \msim x)= \neg_G \neg_G \star x$ stating that $x>1/2$ iff $\star x>0$.
	
	\item[($\star3$)] Recall that for all  $x$, $\msim\Delta(x\Rightarrow_G \msim x)=1$ if $x>1/2$ and it is $0$ otherwise,  while $\msim\Delta x=1$ iff $x\neq 1$. Furthermore, $\msim\Delta(x \Rightarrow_G \star x )=1$ iff $\star x<x$. Therefore the condition $(\star3)$ stands for requiring $\star$ to be a strictly below the identity on the open interval $(1/2,1)$.  
	
	
	\item[$(\star4)$] The term $\Delta(\msim x \Rightarrow_G  x)\wedge \Delta( \msim y\Rightarrow_G y)\wedge \msim \Delta(x\Rightarrow_G y)$ only takes value $0$ or $1$. In particular, it take $1$ iff $x\geq f$, $y\geq f$ and $x>y$. Similarly, $\msim \Delta(\star x\Rightarrow_G \star y)=1$ if $\star x>\star y$. Thus ($\star5$) states a sort of strict monotonicity of $\star$ for positive elements:  for all positive $x,y$, if $x>y$, then $\star x>\star y$. 
	
	\item[$(\star5)$] Similarly to the above point, $\Delta(\msim x\Rightarrow_G  x)\wedge \Delta(\msim y\Rightarrow_G  y)\wedge \Delta(\star x\Leftrightarrow_G \star y)=1$ iff $x, y\geq f$ and $x=y$, while $ \Delta( x\Leftrightarrow_G y)=1$ iff $x=y$. Therefore $(\star 6)$ states that $\star$ is injective on positive elements:  for all positive $x$ and $y$, if $\star x=\star y$, then $x=y$.
\end{itemize}

Now, we rise the question whether the equations introduced in Definition \ref{def:Ginvstar} above are enough for $\star$ to capture the standard behavior of the \luk\ square operator $*$ on $[0,1]$. Equivalently, we are asking if every countable $IG^\star$-chain embeds into the algebra 
\begin{equation}\label{eq:GMV}
[0,1]_{GMV}^*=([0,1], \ast, \To_G, \msim, 0, 1), 
\end{equation}
the expansion of $[0,1]_{MV}^*$ with G\"odel implication, where for all $x,y\in [0,1]$, $\ast x=\max\{0, 2x-1\}$, $x\To_G y=1$ if $x\leq y$ and $x\To_G y=y$ otherwise, and $\msim x=1-x$ is the involution.

By definition, the variety $\mathbb{IG}^\star$ of $IG^\star$-algebras is prelinear. 
We begin investigating the finite linearly ordered algebras of $\mathbb{IG}^\star$. Basic properties are the following:
\begin{enumerate}
	\item If ${\bf A} = (A,  \ast, \msim, 0 ,1)$ is a finite subalgebra of $[0,1]_{MV}^*$, then $\To_G$ is definable in ${\bf A}$ and ${\bf A_G} = (A, \ast,  \msim, \To_G, 0 ,1)$ is a finite chain of $\mathbb{IG}^\star$. However, the variety generated by $[0,1]_{MV}^*$ is not the one generated by 
	$[0,1]_{GMV}^*$, as $\To_G$ is not definable in the infinite chain $[0,1]_{MV}^*$.
	\item The procedure ${\tt P}$ described in Definition \ref{remProcP} can be easily adapted and used so as to  
	define $\langle x \rangle^\star$, the subalgebra defined by an element $x$ in any finite chain of $\mathbb{IG}^\star$.
	\item For every finite IG$^\star$-chain ${\bf A}$ and every $a\in A$, the notion of  skeleton $Sk({\bf A}, a)$ is defined as for  subalgebras of $[0,1]_{MV}^*$ in the previous subsection.
	\item  Proposition \ref{partition} (1) is also valid for finite $IG^\star$-chains.
\end{enumerate}
It is clear that any finite subalgebra of $[0,1]_{MV}^*$ can be embedded into a finite chain of $\mathbb{IG}^\star$. The converse is not true in general as the following examples show. 

\begin{example} 
	Let $\bf A$ be the 6-element $IG^\star$-chain with support $A = \{1, a,b,\msim b, \msim a, 0\}$,  where $0 < \msim a < \msim b < b < a  < 1$ and the operations $\vee, \To_G, \neg_G$ defined according the order and $\star a = \msim b, \star b = \msim a$. This algebra is not embeddable in $[0,1]_{MV}^*$ because both elements $a,b$ satisfy in $\bf A$ the  equation $\ninv {\star} \ninv {\star} (x) = x$, while the corresponding equation in $[0,1]_{MV}^*$, ${\msim}{*} {\msim} {*} (x) = x$, has as a unique solution $x = 2/3$. The algebra generated by $2/3$ in  $[0,1]_{GMV}^*$, $\langle 2/3\rangle^*$, has universe $\{1, 2/3, 1/3, 0\}$ and hence $\langle 2/3\rangle^*$  is not isomorphic to ${\bf A}$. Notice that, by definition of ${\bf A} \in \mathbb{IG}^\star$,  $Sk({\bf A}, a) = Sk({\bf A}, b) = [\star, \ninv, \star, \ninv]$. This sequence is periodic and we have already proved in Proposition \ref{non-periodical} that there is no  strictly simple finite subalgebra of $[0,1]_{MV}^*$ with such a skeleton. 
	
	Also observe that in this algebra it is not possible to define the operators $\Delta_x$ for every $x \in A$  as defined in Proposition \ref{prop3}. In fact the algorithm given in the proof of that proposition does not terminate. 
	This implies that the axiomatization given in Section \ref{sec:logics} for the many-valued logic with semantics on a $\bL_{n+1}^*$-chain is not generalizable to the case of a finite $IG^\star$-chain.
\end{example}

\begin{example} \label{ex14}
	Let $\bf A$ be the 14-element $IG^\star$-chain whose support is $A=\langle a\rangle^\star\cup\langle b\rangle^\star$, where $a>b$ and, for $x = a,b$, $ \langle x \rangle^\star $ is made of the elements 
	$$
	1> x= {\sim} {\star} {\sim} {\star^2} x  > \star x > \msim\star^2 x > {\star^2} x >  {\msim} {\star} x >\star\msim\star^2 x > 0
	$$ 
	as in Figure \ref{figureAyB}.

\begin{center}
\begin{figure}[h!]
\vspace*{-0.25cm}
  \centering
    \includegraphics[width=0.85\textwidth]{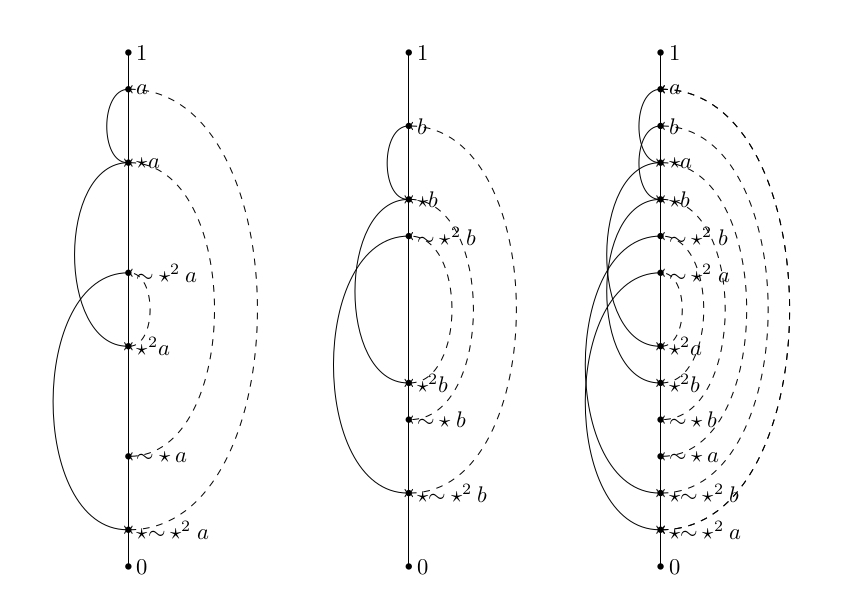} 
		\caption{\small Graphical representation of the algebras $\langle a\rangle^*$,  $\langle b\rangle^*$ and  $\bf A$ from Example \ref{ex14}.  }\label{figureAyB}
	\end{figure}
\end{center}

\vspace*{-0.75cm}

	Operations $\vee, \To_G, \neg_G$ are defined according to the order. An easy computation shows that $\bf A$ is not embeddable into $[0,1]_{MV}^*$ since $ \langle a \rangle^\star$ and  $\langle b \rangle^\star$ are strictly simple and $Sk( \langle a \rangle^\star, a) = Sk( \langle b \rangle^\star, b) = [\star, \star, \sim, \star, \sim]$  and in $[0,1]_{MV}^*$, by Proposition \ref{partition}, there are not two different strictly simple subalgebras with the same skeleton.
\end{example}

In the light of the examples above, let us introduce the following definition.

\begin{definition}
	A finite chain of $\mathbb{IG}^\star$ is called {\em representable}  when its implication free-reduct is embeddable into $[0,1]_{MV}^*$, or in other words, when  it is isomorphic to a finite subchain of $[0,1]_{MV}^*$.
\end{definition}

Representable $IG^\star$-chains ($RIG^\star$-chains for short) form  a proper subset of finite chains of $IG^\star$.  
Our next theorem characterizes the $RIG^\star$-chains. 
Before proving it, we will need a first result that extends Proposition \ref{partition} to finite $IG^\star$-chains. To this end let us point out that, for every $IG^\star$-chain ${\bf A}$ which is not necessarily a subalgebra of $[0,1]^*_{MV}$ and for every $a\in A$, the procedure ${\tt P}$ (Definition \ref{remProcP}) still produces, once launched on $a$, a list of elements of $A$ and it stops when it finds an element $b$ already met at a previous step. Thus, one can easily define, for every $a\in A$ the {\em skeleton} of $a$ in ${\bf A}$, the strictly simple subalgebra $\langle b\rangle^\star$ associated to $a$ and, for every strictly simple subalgebra ${\bf B}$ of ${\bf A}$, the set $B^+$ as we did before Proposition \ref{partition}. Then the following holds.

\begin{proposition}\label{Newpartition}
	Let ${\bf A}$ be a finite $IG^\star$-chain and let ${\bf B}_1, {\bf B}_2,\ldots, {\bf B}_k$ the strictly simple subalgebras of ${\bf A}$. Then $\{B_1^+, ...B_k^+\}$ is a partition of $A\setminus\{0,1\}$. Furthermore,  if ${\bf B}_i$ has a non-periodic skeleton, ${\bf B}_i$ is representable and each $B_i^+$, regarded as partial algebra, partially embeds into $[0,1]^*_{MV}$.
\end{proposition} 
\begin{proof}
	The first part of the claim is proved, with no modification, by the same proof of Proposition \ref{partition}. Indeed, in that proof, no assumption on the fact that ${\bf A}$ is subalgebra of $[0,1]^*_{MV}$ is made and hence it perfectly applies to this more general case.
	
	As for the second part of the statement, assume that ${\bf B}_i$  has a non-periodic skeleton. Thus, in particular $ Sk({\bf B}_i,\co_i)$ for $\co_i$ being the coatom of $ {\bf B}_i$. Thus, the equation $f_{Sk({\bf B}_i, \co_i)}(x)=x$ has a unique rational solution $r$ in $[0,1]_{MV}^*$. It is then easy to see that the finite subalgebra $\langle r\rangle^*$ of $[0,1]_{MV}^*$ is indeed isomorphic to ${\bf B}_i$ and the assignment $\lambda: b\mapsto r$ determines an embedding of ${\bf B}_i$ into $[0,1]_{MV}^*$. 
	
	Finally, in order to partially embed the partial algebra $B_i^+$ into $[0,1]_{MV}^*$ recall that $B_i^+= \{a \in A : \langle a \rangle \supseteq B_i\}$, or equivalently, $ B_i^+ = B_i \cup \{a\in A\mid f_R(a)\in B_i$ for some finite  sk-sequence  $R\}$. Since we already showed that ${\bf B}_i$ embeds into $[0,1]_{MV}^*$, it is left to show how to map the elements $a$'s such that $f_R(a)=b_a\in B_i$ for some sequence $R$. Since ${\bf B}_i$ embeds, through a mapping $\lambda$, into the rational subalgebra of $[0,1]_{MV}^*$, the equation $f_R(x)=\lambda(b_a)$ has a rational solution, say $r_a$. Then, extend $\lambda$ to a mapping sending each $a$ of the above kind to $r_a$. The so obtained map clearly is a partial embedding of ${\bf B}_i^+$ into $[0,1]_{MV}^*$.
\end{proof}

Now, we are ready to characterize the representable $IG^\star$-chains.

\begin{theorem}\label{thmRepresentable}
	A finite $IG^\star$-chain ${\bf A}$ is representable iff
	\begin{enumerate}
		\item For any strictly simple subalgebra ${\bf B}$ of $\bf A$ and for any positive $b\in B\setminus\{1\}$, $Sk( {\bf B},b)$ is non-periodic,
		\item For each pair of strictly simple subalgebras ${\bf B}$ and ${\bf C}$ of ${\bf A}$ there are no positive elements $b\in B\setminus\{1\}$ and $c\in C\setminus\{1\}$ such that $Sk( {\bf B}, b)=Sk({\bf C}, c)$.
	\end{enumerate}
\end{theorem}
\begin{proof}
	Left-to-right. If ${\bf A}$ is representable, then it is (isomorphic to) a subalgebra of $[0,1]^*_{MV}$. Therefore, (1) and (2) immediately follow from Proposition \ref{non-periodical} and Proposition \ref{partition} (2) respectively.
	\vspace{.2cm}
	
	\noindent Right-to-left. Assume (1) and (2) hold. (1) implies, by Proposition \ref{Newpartition} that, for each strictly simple subalgebra ${\bf B}$ of ${\bf A}$ the partial algebra $B^+$ partially embeds into the rational subalgebra of $[0,1]_{MV}^*$ and hence it embeds into an $\bL^*_{n_B+1}$ for some natural number $n_B$. Moreover, (2) implies that for two different strictly simple subalgebras ${\bf B}$ and ${\bf C}$ of ${\bf A}$, $B^+$ and $C^+$ do not partially embed into the same $\bL^*_{n+1}$. In other words, for every strictly simple subalgebra ${\bf B}$ of ${\bf A}$ there exists a unique $n_B$ and a unique partial embedding $\lambda_B$ of $B^+$ into $\L^*_{n_B+1}$. Let $k={\rm lcm}\{n_B\mid {\bf B}\mbox{ is a strictly simple subalgebra of }{\bf A}\}$. Thus, each $B^+$ partially embeds into $\bL^*_{k+1}$ by the same map $\lambda$ which, adding $\lambda(0)=0$ and $\lambda(1)=1$ determines and embedding of ${\bf A}$  into $\bL^*_{k+1}$.
\end{proof}
A direct inspection on the proof of Theorem \ref{thmRepresentable} above suggests that points 1 and 2 of its statement can be equationally described. Indeed, in the following result, we will prove that for every $n$, representable $IG^\star_{n+1}$-algebras form a proper subvariety of $\mathbb{IG}^\star_{n+1}$.

In order to see it consider, for all $n\in \mathbb{N}$, for all
sk-sequences $R=[o_1,\ldots, o_t]$ and for all natural numbers $r$ such that $rt\leq n+1$, the following equations: 

%

\begin{itemize}
	\item[$(R1n)$] $\ninv x \vee x \vee (x \TTo_G y)\; \vee [\Delta(f_{R} (x) \TTo_G y) \To_G \ninv\Delta(f_{(r-1)R} (y) \TTo_G x)] = 1$;
	\item[$(R2n)$] $\ninv x  \vee x \vee \ninv y  \vee  y \; \vee  [(\Delta (f_{R} (x) \TTo_G x) \wedge (f_{R}(y) \TTo_G y)) \To_G \Delta(x \TTo_G y)] = 1$.
	
\end{itemize}

\begin{theorem}\label{axiomatic}
	Let ${\bf A}$ be a finite $IG^\star$-algebra such that its $G$-reduct belongs to $\mathbb{G}(n+1)$. Then ${\bf A}$ is representable iff, for all
	sk-sequences $S=[o_1,\ldots, o_k]$ and for all natural numbers $r$ such that $rk\leq n+1$, ${\bf A}$ satisfies $(R1n)$ and $(R2n)$.
\end{theorem}
\begin{proof}
	(Left-to-right). Assume ${\bf A}$ is not representable. Then, by Theorem \ref{thmRepresentable}, either: (1) ${\bf A}$ has a strictly simple subalgebra ${\bf B}$ such that $Sk({\bf B},b)$ is periodic, for a positive $b\in B\setminus\{1\}$, or (2) ${\bf A}$ has two strictly simple subalgebras ${\bf B}$ and ${\bf C}$ with positive elements $b\in B\setminus\{1\}$ and $c\in C\setminus\{1\}$ such that $Sk({\bf B}, b)=Sk({\bf C},c)$. 
	
	Assume  that (1) is the case and let $S=[o_1,\ldots, o_k]$ be the periodic skeleton of $b$ in ${\bf B}$. Then there is an initial non-periodic sk-subsequence $R=[o_1,\ldots, o_t]$ of $Sk(b, {\bf B})$ and a natural number $r$ such that $[o_1,\ldots, o_k]$ is the repetition $r$-times of $[o_1,\ldots, o_t]$, that is, $S=(r)R $. Call $c=f_{R}(b)$. Thus we have that $\msim b < 1$, $b < 1$ and $b\TTo_G c < 1$ 
	On the other hand $\Delta(f_{R}(b)\TTo_Gc) = 1$ holds by definition of $c$, and also $\Delta(f_{(r-1)R}(c)\TTo_G b) = 1$  holds because $(r)[o_1,\ldots, o_t]=[o_1,\ldots, o_k]$ is the skeleton of $b$. Thus, $\neg\Delta(f_{(r-1)R}(c)\TTo_Gb) = 0$  and hence $(R1n)$ is not satisfied. 
	
	Hence, assume that (2) is the case. Since ${\bf B}$ and ${\bf C}$ are both strictly simple, $B{\setminus}\{0,1\}\cap C{\setminus}\{0,1\}=\emptyset$. Take positive elements $b\in B\setminus\{1\}$ and $c\in C\setminus\{1\}$. By hypothesis $Sk({\bf B},b)=Sk({\bf C},c)=S=[o_1,\ldots, o_k]$. Then one has $\Delta((f_{S}(b)\TTo_Gb)\wedge (f_{S}(c)\TTo_Gc)) = 1$ while $\Delta(b \TTo_G c) = 0$. This shows that $(R2n)$ fails as well.
	\vspace{.05cm}
	
	(Right-to-left). Let us assume that there exists a non-periodic sk-sequence $S=[o_1,\ldots,o_k]$ and a natural number $r$ such that $rk\leq n$ and either $(R1n)$ fails or $(R2n)$ fails.
	
	If $(R1n)$ fails, then there exist $x,y\in A$ different from $0$ and $1$ such that  $x\neq y$, $f_{R}(x)\TTo_Gy=1$ and $f_{(r-1)R}(y)\TTo_Gx=1$. Thus $f_{(r)R}(x)\TTo_Gx=1$, meaning that the subalgebra $\langle x\rangle^\star$ generated by $x$ has a periodic skeleton. 
	Thus ${\bf A}$ is not representable by Theorem \ref{thmRepresentable}.
	
	If $(R2n)$ fails, then there are two distinct positive elements $b, c\in A\setminus\{1\}$ having the same skeleton. The strictly simple subalgebras $\langle b\rangle^\star$ and $\langle c\rangle^\star$ of ${\bf A}$ witness the fact that ${\bf A}$ is not representable again by Theorem \ref{thmRepresentable}.
\end{proof}

\begin{remark}
	(1) As it was observed after the Proposition \ref{prop3}, in any finite RIG$^\star$-chain $\bf A$  it is possible to define the operators $\Delta_x$ for every $x \in A$. This implies that the axiomatization of the variety generated by the chain $\bL_{n+1}^*$ given in Definition \ref{def:GInvStarAlgebras} and the proof of Theorem \ref{thm:GenerateLambda}  can be easily generalized to axiomatize the variety generated by a single finite RIG$^\star$-chain.
	\vspace{.2cm}
	
	\noindent(2) Theorem \ref{axiomatic} gives an axiomatization of the variety generated by the representable $IG^\star$-chains whose length is less or equal to $n+1$. This is the axiomatization of a variety generated by a finite family of chains, very different from the axiomatization in Definition \ref{calLin} that gives the axiomatization of the variety generated by a single  $RIG^\star$-chain. 
\end{remark}

Now we know that the answer to the question posed after Definition \ref{def:Ginvstar} is negative and we reformulate the question as whether every countable $RIG^\star$-chain embeds into the algebra $[0,1]^*_{GMV}$.
In the next result we will denote by $\mathbb{RIG}^\star$ the variety generated by the finite representable $\mathbb{IG}^\star$-chains while $\mathbb{V}^\star$ will denote the variety generated by $[0,1]^*_{GMV}$.

\begin{proposition} The following statements are valid:
	\begin{enumerate}
		\item The variety $\mathbb{V}^\star$ has the finite model property.
		\item  The varieties $\mathbb{V}^\star$ and $\mathbb{RIG}^\star$ coincide
		\item The variety $\mathbb{V}^\star$ is axiomatized by the axioms of $\mathbb{IG}^\star$ plus the infinite set of axioms (R1n) and (R2n) for every $n \geq 2$.
	\end{enumerate}
\end{proposition}

\begin{proof}
	To prove (1), suppose that $\varphi$ is not a tautology in $\mathbb{V}^\star$. Then there is an evaluation $e$ to the  chain $[0,1]^*_{GMV}$ such that $e(\varphi) < 1$. Then there is also a rational evaluation $v$ (that is a good approximation of $e$) such that $v(\varphi) < 1$ and for all propositional variable $p$ appearing in $\varphi$, $e(p)$ is rational. Since the subalgebra generated by the set of values $\{v(p)$:  $p$ is a propositional variable appearing in $\varphi\} \subseteq [0, 1]$ is finite, $\varphi$ is not valid in a finite $RIG^\star$-chain.
	
	On the other hand, (2) is immediate from (1) since both varieties are generated by finite subalgebras of $[0,1]^*_{GMV}$ which are the representable $\mathbb{IG}^\star$-chains.
	
	Finally, (3) is a direct consequence of (2).
\end{proof}

\section{Conclusions and final remarks} \label{sec:conclusions}

In this paper we have been concerned with the logical and algebraic  analysis of the reduct of finite-valued {\L}ukasiewicz logics over the signature $(\lor, \land, \neg, * )$, where $*$ represents the square operator $* x = x \odot x$, with $\odot$ being {\L}ukasiewicz strong conjunction.  Our main contributions are the following. 
First of all, we have characterized for which  $n$ of the corresponding structures $\L_{n+1}^*$, over the $(n+1)$-element domain $\{0, 1/n, \ldots, 1\}$, the {\L}ukasiewicz implication is definable, and thus for which $n$ the algebra $\L_{n+1}^*$  is term-equivalent to the MV-chain $\L_{n+1}$. 
Second, we have studied the matrix logics arising from the $\L_{n+1}^*$ structures with order filters. We have shown they are all algebraizable, we have described the resulting varieties of $\Lambda_{n+1}^*$-algebras that constitute their equivalent algebraic semantics, and provided a complete and uniform Hilbert-style axiomatisation in a suitable signature that enjoys nice logical properties. 
And third, we have considered an alternative approach to capture the behaviour of the square operator in algebraic structures obtained by adding a unary operator $\star$ to $n$-valued G\"odel chains with an involutive negation, and have identified the conditions under which they can be embedded into some $\L_m^*$.

%
%
%
%

At this point we would like to make a couple of additional remarks we deem interesting to highlight. 
%
%
%
%
An interesting question is whether the well-known relationship between the finite-valued logics $\L_n$ and the $[0, 1]$-valued logic $\L$  is preserved  between the logics $\Lambda_n^*$ and their corresponding $[0, 1]$-valued version $\Lambda^*$.  It is well-known that, with respect to their finitary consequence relations,  $\L$ is the intersection of all the finite-valued logics  $\L_n$, i.e. $\bigcap_n \L_n = \L$. It is not difficult to  check that this relationship extends to our setting as follows:

\begin{itemize}
	\item[-]  In the signature $(\neg, *, \lor)$, we  have that $\bigcap_n \Lambda_n^* = \Lambda^*$, where  $\Lambda^*$ is the matrix logic $\langle [0, 1]_{MV}^*, \{1\}\rangle$ defined in the obvious way similarly as the logics $\Lambda_n^*$. 
	
	\item[-] In the expanded signature $(\neg, *, \lor, \Rightarrow_G )$, we also have $\bigcap_n \Lambda_n^* = \Lambda^{*,\Rightarrow}$, where now $\Lambda^{*,\Rightarrow} = \langle [0, 1]_{GMV}^*, \{1\}\rangle$. 
	
\end{itemize}
Another way to look at the relation $\bigcap_n \L_{n} = \L$ is that \luk\ logic is complete with respect to the whole class of finite MV-chains. From the results of Subsection \ref{subsec:primes}, we know that, if $n$ is a prime number in $\Pi$ (recall Definition \ref{def:campinas}), then the algebras $\bL_{n+1}$ and $\L^*_{n+1}$ are term equivalent. Thus, we rise the question of whether primes from $\Pi$ are enough to define a complete semantics for \L. In order words, we will study if \L\ is complete with respect to the set of finite chains $\bL^*_{n+1}$ where $n\in \Pi$. Clearly, in order to provide an answer to the question above, we would need first to prove if $\Pi$ is an infinite set or not. 

Our future work in this topic will also concern the variety $\mathbb{IG}^\star$ that we introduced in Subsection \ref{subsec:IGStar}. In particular we will investigate whether $\mathbb{IG}^\star$ can be generated by {\em standard algebras}, that is to say, by $IG^\star$-chains based on the real unit interval and if, moreover, $\mathbb{IG}^\star$ can be generated by its finite chains. The latter, then, would give the finite model property for its associated logic.

%
%
%

\subsection*{Acknowledgments}  Coniglio acknowledges support from  the  National Council for Scientific and Technological Development (CNPq), Brazil
under research grant 306530/2019-8.  Esteva, Flaminio and Godo acknowledge partial support by the Spanish project 
PID2019-111544GB-C21. Flaminio also acknowledges partial support by the Spanish Ram\'on y Cajal research program RYC-2016-19799. 


\end{document}